\documentclass[12pt]{amsart}
\usepackage{a4wide}
\usepackage{amssymb,amsfonts,amsmath,amsthm}
\usepackage{latexsym}
\usepackage[latin1]{inputenc}

\usepackage[hyperindex,colorlinks]{hyperref}

\newcommand{\e}{{\mathord{\bf e}}}
\newcommand{\pa}{{\mathord{\partial}}}

\newcommand{\setN}{{\mathord{\mathbb N}}}

\newcommand{\setR}{{\mathord{\mathbb R}}}

\def\supp{\text{\rm{supp}}}

 \newtheorem{thm}{Theorem}[section]
 \newtheorem{cor}[thm]{Corollary}
 \newtheorem{lem}[thm]{Lemma}
 \newtheorem{prop}[thm]{Proposition}
 \theoremstyle{definition}
 \newtheorem{defn}[thm]{Definition}
 \theoremstyle{remark}
 \newtheorem{rem}[thm]{Remark}
 
 \numberwithin{equation}{section}

\newtheorem{con}[thm]{Assumptions}

\begin{document}

\title[Local existence of relativistic perfect
fluids]
 {Local existence of solutions of  self gravitating relativistic perfect
fluids}
%----------Author 1
\author[U. Brauer]{Uwe Brauer}

\address{%
  Departamento de  Matem\'atica Aplicada\\ Universidad Complutense Madrid
28040 Madrid, Spain}
\email{oub@mat.ucm.es}
    
\author[L. Karp]{Lavi Karp}

\address{%
Department of Mathematics\\ ORT Braude College\\
P.O. Box 78, 21982 Karmiel\\ Israel}

\email{karp@braude.ac.il}

\thanks{Research  supported  ORT
Braude College's Research Authority}
%----------Author 2

%----------classification, keywords, date


\begin{abstract}

  This paper deals with the evolution of the Einstein gravitational
  fields which are coupled to a perfect fluid.  We consider the
  Einstein--Euler system in asymptotically flat spacestimes and
  therefore use the condition that the energy density might vanish or
  tend to zero at infinity, and that the pressure is a fractional
  power of the energy density. In this setting we prove a local in
  time existence, uniqueness and well posedness of classical
  solutions.  The zero order term of our system contains an expression
  which might not be a $C^\infty$ function and therefore causes an
  additional technical difficulty.  In order to achieve our goals we
  use a certain type of weighted Sobolev space of fractional order.
  In \cite{BK3} we constructed an initial data set for these of
  systems in the same type of weighted Sobolev spaces.

  We obtain the same lower bound for the regularity  as Hughes,
  Kato and Marsden \cite{hughes76:_well} got for the vacuum Einstein
  equations. However, due to the presence of an equation of state with
  fractional power, the regularity is bounded from above.

\end{abstract}

%%% ----------------------------------------------------------------------
\maketitle
%%% ----------------------------------------------------------------------
%\tableofcontents

\section{Introduction}
\label{sec:introduccion}

This paper deals with the Cauchy problem for the Einstein--Euler
system describing a relativistic self-gravitating perfect fluid, whose
density either has compact support or falls off at infinity in an
appropriate manner.

The evolution of the gravitational field is described by the Einstein
equations
\begin{equation}
\label{eq:eineul:1}
G_{\alpha\beta}= R_{\alpha\beta}-
\frac{1}{2}g_{\alpha\beta}R = 8\pi T_{\alpha\beta},
\end{equation}
where $g_{\alpha\beta}$ is a semi Riemannian metric having a
signature $(-,+,+,+)$, $R_{\alpha\beta}$ is the Ricci curvature tensor, and $R$
is the scalar curvature.
Both tensors are functions of the metric $ g_{\alpha\beta}$ and its first and
second order partial derivatives.
The right-hand side of (\ref{eq:eineul:1}) consists of the
energy--momentum tensor $T_{\alpha\beta}$, which in the case of a perfect fluid
takes the form
\begin{equation}
  \label{eq:eineul:2} 
T^{\alpha\beta} = (\epsilon+p) u^\alpha u^\beta + pg^{\alpha\beta},
\end{equation} 
where $\epsilon$ is the energy density, $p$ is the pressure and $u^\alpha$ is the
four-velocity vector.
The vector $u^{\alpha}$ is a unit timelike vector, which means that it
satisfies the normalization condition

\begin{equation}
  \label{eq:publ-broken:2} g_{\alpha\beta} u^\alpha u^\beta=-1.
\end{equation} 
The Euler equations describing the evolution of the
fluid take the form
\begin{equation}
  \label{eq:eineul:3} \nabla_\alpha T^{\alpha\beta}=0,
\end{equation} 
where $\nabla$ denotes the covariant derivative associated
with the metric $g_{\alpha\beta}$.
Equations (\ref{eq:eineul:1}) and (\ref{eq:eineul:3}) are not
sufficient to determinate the structure uniquely, a functional
relation between the pressure $p$ and the energy density $\epsilon$ (equation
of state) is also needed.
We choose an equation of state that has been used in astrophysical
problems.
It is the analogue of the well known polytropic equation of state in
the non-relativistic theory, given by
\begin{equation}
  \label{eq:eineul:4} p = p(\epsilon) = K\epsilon^{\gamma}, \qquad K,\gamma\in\setR^{+},\qquad 1<\gamma.
\end{equation}
The sound velocity is denoted by
\begin{equation*}
 \sigma^2=\frac{d p}{d\epsilon},
\end{equation*}
and the range of the energy density $\epsilon$ will be
restricted so that the causality condition
\begin{math}
 \sigma^2<1
\end{math} 
will hold.

The unknowns of these equations are the semi Riemannian metric
$g_{\alpha\beta}$, the velocity vector $u^\alpha$ and the energy density $\epsilon$.
These are functions of $t$ and $x^a$, where $x^a$ $(a=1,2,3)$ are the
Cartesian coordinates on $\setR^3$.
The alternative notation $x^0=t$ will also be used and Greek indices
will take the values $0,1,2,3$ in the following.

In the present paper we prove the well-posedness of the coupled
systems (\ref{eq:eineul:1}), (\ref{eq:eineul:2}), (\ref{eq:eineul:3})
and (\ref{eq:eineul:4}) under the harmonic gauge condition in
asymptotically flat spacetimes.
In order to achieve this, we need to rewrite the above equations as a
hyperbolic system.

In astrophysical context the density $\epsilon$ is expected to have compact
support, or tend to zero at spatial infinity in an appropriate sense.
It is well known that the usual symmetrization of the Euler equations
is badly behaved in cases where the density tends to zero somewhere.
The coefficients of the system degenerate or become unbounded when $\epsilon$
approaches zero.
It was observed by Makino
\cite{makino86:_local_exist_theor_evolut_equat_gaseous_stars} that
this difficulty can be to some extend circumvented in the case of a
non-relativistic fluid by using a new matter variable $w$ in place of
the mass density.
For this reason we introduce the quantity
\begin{equation}
  \label{eq:intro:11} 
w=M(\epsilon)=\epsilon^{\frac{\gamma-1}{2}},
\end{equation} 
and we call it the Makino variable.
A similar device was used by Gamblin \cite{gamblin93:_solut_euler} and
Bezard \cite{Bezard_93} for the Euler-Poisson equations, and by
Rendall \cite{Rendall92:-fluid} and Oliynyk
\cite{oliynyk07:_newton_limit_perfec_fluid} for the Einstein--Euler
equations.
The common method for solving the Cauchy problem for the Einstein
equations consists usually of the following steps.
\begin{enumerate}
  \item[{ 1.}] Initial data must satisfy the constraint equations,
which are intrinsic to the initial hypersurface. Therefore, the first
step is to construct solutions of these constraints.
  \item[{ 2.}]  The second step is to use the harmonic coordinate
condition and to solve the evolution equations with these initial
data.
  \item[{ 3.}]  The last step is to prove that the harmonic coordinate
condition and the solution of the constraints propagate. That means if
they held on a initial hypersurface, they hold for later times.
\end{enumerate}
The last step was treated in detail, for example in Fisher and
Marsden\cite{FMA}.
The idea is to work out the condition $\nabla_{\alpha}G^{\alpha\beta}=0$.
Since our energy--momentum satisfy (\ref{eq:eineul:3}), their result
can be immediately generalised for our case.
But for the sake of brevity we have omitted the details.

However, the presence of the equation of state (\ref{eq:eineul:4})
introduces an additional step: the compatibility problem of the
initial data for the fluid and the gravitational field (see
(\ref{eq:intro:7})).
There are three types of initial data for the Einstein--Euler system:

\begin{itemize}
  \item The gravitational data is a triple $(\mathcal{M},h,K_{ab})$,
  where $\mathcal{M}$ is a space-like manifold, $h$ is a proper
  Riemannian metric on $\mathcal{M}$, and $K_{ab} $ is the second
  fundamental form on $\mathcal{M}$ (extrinsic curvature).
  The pair $(h,K_{ab})$ must satisfy the constraint equations
  (\ref{eq:intro:6});

  \item The matter variables, consisting of the energy density $z$ and
  the momentum density $j^a$, appear on the right hand side of the
  constraints (\ref{eq:intro:6});

  \item The initial data for Makino variable $w$ and the velocity
  vector $u^\alpha$ of the perfect fluid.

\end{itemize}
The only type of Sobolev spaces which are known to be useful for
existence theorems for the constraint equations in an asymptotically
flat manifold, are the weighted Sobolev spaces $H_{k,\delta}$, where $k\in\setN $
and $\delta\in\setR$.
These spaces were introduced by Nirenberg and Walker
\cite{nirenberg73:_null_spaces_ellip_differ_operat_r} and Cantor
\cite{cantor75:_spaces_funct_condit_r}, and they are the completion of
$C_0^\infty(\mathbb{R}^3)$-functions under the norm
\begin{equation}
\label{eq:intro:5} 
\|u\|_{k,\delta}^{2} =\sum_{|\alpha|\leq k}\int\left((1+| x|)^{\delta+|\alpha|}
|\partial^\alpha u|\right)^2dx.
\end{equation}

Due to the presence of the equation of state (\ref{eq:eineul:4}) and
the Makino variable (\ref{eq:intro:11}), we have to estimate
$\|w^{\frac{2}{\gamma-1}}\|_{{k,\delta}}$.
So it is perhaps worth discussing the estimate of the Sobolev's norm
of $u^\beta$ in more details for $\beta>1$.
For simplicity we discuss this in the ordinary Sobolev space
$H^k=H^k(\setR^3)$.
The simplest case is when $\beta\in \setN $, then $\|u^\beta\|_{H^k}\leq C(\|u\|_{L^\infty})
\|u\|_{H^k}$ and there is no restriction on $k$.
When $\beta\not\in\mathbb{N}$, then we obtain the same estimate, provided
that $k\leq \beta$.
This bound on $k$ was improved by Runst and Sickel
\cite{runst96:_sobol_spaces_fract_order_nemyt} to $k<\beta+\frac{1}{2}$.
Applying this to $\beta=\frac{2}{\gamma-1}$, and taking into account the
Sobolev embedding $\|u\|_{L^\infty}\leq C\|u\|_{H^k}$ for $k>\frac{3}{2}$, we get
a lower and upper bound for $k$:
\begin{equation}
\label{eq:intro:1} \frac{3}{2}<k<\frac{2}{\gamma-1}+\frac{1}{2}.
\end{equation} 
The only exception is the case when $\frac{2}{\gamma-1}$ is an integer.
Note that for certain values of $\gamma$, inequalities (\ref{eq:intro:1})
possesses no integer solution.
Hence, for these values of $\gamma$ it is impossible to obtain a solution
to the Einstein--Euler system in Sobolev spaces of integer order.
So in order to be able to solve the coupled system for the maximal
range of the power $\gamma$, and in addition, to improve the regularity of
the solutions, we are considering the Cauchy problem in the weighted
fractional spaces $H_{s,\delta}$, where $s$ is real number (see Definition
\ref{def:weighted:3}).
These spaces were introduced by Triebel
\cite{triebel76:_spaces_kudrj2}, and they generalize $H_{k,\delta}$ to a
fractional order.
In \cite{BK3} the authors constructed initial data for coupled systems
(\ref{eq:eineul:1}), (\ref{eq:eineul:2}) and (\ref{eq:eineul:3}) with
the equations of state (\ref{eq:eineul:4}).
This includes the solution to the constraint equations
(\ref{eq:intro:6}), as well as the solution to the compatibility
problem between the matter variable $(z,j^a) $ and the fluid variables
$(w,u^\alpha)$, (\ref{eq:intro:7}), in the $H_{s,\delta}$-spaces.
Here we will establish the well-posedness of Einstein--Euler systems
in the weighted fractional Sobolev spaces $H_{s,\delta}$.

The common way to prove well-posedness is to rewrite the evolution
equations as a symmetric hyperbolic system.
So our first step is to use the Makino variable
(\ref{eq:intro:11}) and to reduce the Euler equations
(\ref{eq:eineul:3}) to a uniformly first order symmetric hyperbolic
system.
This result was announced in \cite{ICH12} and here we present a
detailed proof of it.
Our hyperbolic reduction is based on the fluid decomposition; for
alternative reductions see \cite{Rendall92:-fluid}.

It is well-known that the Einstein equations can be written as a
system of quasi-linear wave equations under the harmonic gauge
condition \cite{choquet--bruhat52, CHY, wald84:_gener_relat}.
The proofs of existence and uniqueness either use second order
techniques \cite{choquet--bruhat52, Christodoulou_1981, HAE,
  hughes76:_well, KR1}, or transferring the equations to a first order
symmetric hyperbolic system.
Fischer and Marsden used the first order techniques and obtained the
well-posedness of the reduced vacuum Einstein equations in $ H^s$ and
for $s>\frac{7}{2}$ \cite{FMA}.
This result was improved by Hughes, Kato and Marsden
\cite{hughes76:_well}, who obtained 
$(g_{\alpha\beta},\partial_t g_{\alpha\beta})\in H^{s+1}\times
H^s$ for $s>\frac{3}{2}$.
They used second order theory, and took advantage of the specific form
of the quasi-linear system of wave equations, namely, that the coefficients
depend only on the semi-metric $g_{\alpha\beta}$, but not on its first order
derivatives.

Our aim is to prove existence and uniqueness of the reduced
Einstein--Euler system (\ref{eq:eineul:1}), (\ref{eq:eineul:2}) and
(\ref{eq:eineul:3}) with the equation of state (\ref{eq:eineul:4}).
In addition, we would like to achieve the same regularity of the
metric as in \cite{hughes76:_well}.
But since we have here a coupled system which one of them is a first
order, the second order techniques of Hughes, Kato and Marsden in
\cite{hughes76:_well} are no longer available for the present
problem.

In asymptotically flat spacetimes the initial metric $g_{\alpha\beta}(0)$
differs from the Minkowski metric by a term which is $O(1/r)$ at
spatial infinity, and this term does not belong to $H^s$.
It is therefore more appropriate to consider both the constraint and
evolution equations in the $H_{s,\delta}$ spaces rather than the unweighted
spaces $H^s$.
For the vacuum equations the second author obtained well-posedness of
the reduced Einstein equations with 
$(g_{\alpha\beta},\partial_t g_{\alpha\beta})\in H_{s+1,\delta}\times
H_{s,\delta+1}$, $s>\frac{3}{2}$ and $\delta>-\frac{3}{2}$, see \cite{Ka1}.
But unlike Hughes, Kato and Marsden \cite{hughes76:_well}, he
treated the quasi-linear system of wave equations as a first order symmetric
hyperbolic system.
The first order techniques have the advantage that they enable, in a
convenient way, the coupling of the gravitational field equations to
other matter models, in particular, to perfect fluids.
In the Appendix we explain the main idea of \cite{Ka1} which allows us
to obtain the regularity index $s>\frac{3}{2}$ by means of first order
hyperbolic systems.

A crucial step in the proof of existence and uniqueness of any
hyperbolic system is to establish energy estimates for the linearized
system.
In order to achieve this we define an appropriated inner--product of
the $H_{s,\delta}$ spaces, which takes into account the coefficients of the
linearized system (see Section \ref{sec:energy-estimates} ).
A similar inner--product was used in \cite{Ka1}, and here we rely on
these energy estimates.

Once we have obtained the energy estimates for the linearized system,
we use Majda's iterative scheme in order to obtain existence and
uniqueness of the quasi--linear symmetric hyperbolic system
\cite{majda84:_compr_fluid_flow_system_conser}.
This procedure uses the fact that solutions to a linear first order
symmetric hyperbolic system with $C_0^\infty$ coefficients and initial
data, are also $C_0^\infty$.
But here we encounter a further difficulty, namely, the right hand
side of (\ref{eq:eineul:1}) contains the fractional power
$w^{\frac{2}{\gamma-1}}$, see (\ref{eq:2.3}).
So even when $w\in C_0^\infty$ and $w\geq 0$, $w^{\frac{2}{\gamma-1}}$ might not be a
$C^\infty$ function.
We solve that problem by using the fact  that $\epsilon
=w^{\frac{2}{\gamma-1}}$ satisfies a certain first order linear equation.
Gamblin encountered a similar problem for the Euler--Poisson equations
\cite{gamblin93:_solut_euler}, but he solved it in a somewhat
different way.

Our results improve the existence theory of solutions locally in time
of self gravitating relativistic perfect fluids in several aspects.
Rendall studied this problem in \cite{Rendall92:-fluid}, but he
assumed time symmetry, which means that the extrinsic curvature of the
initial manifold is zero, and therefore the Einstein constraint
equations are reduced to a single scalar equation.
In addition, he dealt only with $C^{\infty}_0$-solutions.
In his study of the Newtonian limit of perfect fluids, Oliynyk
obtained existence locally in time in the weighted space of integer
order $H_{k,\delta}$, for $k\geq 4$
\cite{oliynyk07:_newton_limit_perfec_fluid}.
Both Rendall and Oliynyk assume that the adiabatic exponent of
(\ref{eq:eineul:4}) satisfies the condition that $\frac{2}{\gamma-1}$ is an
integer.
 
The paper is organized as follows: In the next section we define the
weighted Sobolev spaces of fractional order $H_{s,\delta}$ and state the
main result.
Section \ref{sec:init-value-probl} has two subsections: the first one
deals with the hyperbolic reduction of the Euler equations
(\ref{eq:eineul:3}); in the second one we spell out the matrices which
describe the coupled equations (\ref{eq:eineul:1}),
(\ref{eq:eineul:2}) and (\ref{eq:eineul:3}) as a hyperbolic system.

In Section \ref{sec:properties} we present tools and properties of the
$H_{s,\delta}$-spaces which we need in the course of the publication.
We also define the corresponding product spaces.
The energy estimates for the linearized system are considered in
Section \ref{sec:energy-estimates}, there we also define the
appropriate inner-product.
In Section \ref{sec:iteration} we treat the iteration procedure.
Parts of the steps are standard and known, but some of them require
special attention due to the specific form of the system
(\ref{eq:3.4}) and the product spaces.
In this Section we will use the fact that the coefficients of the
first order derivatives depend only on the semi-metric $g_{\alpha\beta}$.
Finally, in Section \ref{sec:main} we prove the main result.
In the Appendix we give a heuristic idea explaining how that fact that
the coefficients of the system of wave equations depend only on the semi--metric
$g_{\alpha\beta}$ enable us to obtain the desired regularity by means of
symmetric hyperbolic systems.

\section{The main results}
\label{sec:main-results}
We obtain the well-posdeness in the weighted Sobolev spaces of
fractional order.
So we first recall their definition.

Let $\{\psi_j\}_{j=0}^\infty \subset C_0^\infty(\setR^3)$ be a sequence of cutoff function such
that, $\psi_j(x)\geq0$ for all $j\geq 0$, $\supp(\psi_j)\subset \{x: 2^{j-2}\leq |x| \leq
2^{j+1}\}$, $\psi_j(x)=1$ on $\{x: 2^{j-1}\leq |x| \leq 2^{j}\}$ for $j=1,2,...$,
$\supp(\psi_0)\subset\{x:|x|\leq 2\}$, $\psi_0(x)=1$ on $\{x: |x|\leq 1\}$ and
\begin{equation*}
|\partial^\alpha  \psi_j(x)|\leq  C_\alpha  2^{-|\alpha|j},
\end{equation*}
where the constant $C_\alpha$ does not depend on $j$.

We restrict ourselves to the case $p=2$ and denote the Bessel
potential spaces by $H^s$ with the norm given by
\begin{equation*}
  \|u\|_{H^s}^2=\int(1+|\xi|^2)^s |\hat{u}(\xi)|^2d\xi,
\end{equation*}
where $\hat{u}$ is the Fourier transform of $u$.
\begin{defn}
  \label{def:weighted:3}
  For $s,\delta\in\setR$,
  \begin{equation}
    \label{eq:weighted:4}
    \left(\|u\|_{H_{s,\delta}}\right)^2=
    \sum_{j=0}^\infty  2^{( \frac{3}{2} + \delta)2j} \| (\psi_j
    u)_{(2^j)}\|_{H^{s}}^{2},
  \end{equation}
  where $f_\varepsilon(x)=f(\varepsilon x)$ denotes the scaling by a positive number $\epsilon$.
  The space $H_{s,\delta}$ is the set of all 
  tempered distributions having a finite norm given by
  (\ref{eq:weighted:4}).
\end{defn}
The $H_{s,\delta}$-norm of a distribution $u$ in an open set
$\Omega\subset\setR^3$ is given by
\begin{equation*}
 \left\|u \right\|_{H_{s,\delta}(\Omega)}=\inf\limits_{f{_{\mid_\Omega}}=u}
\left\|f \right\|_{H_{s,\delta}(\setR^3)}.
\end{equation*}
\begin{defn}
\label{def:asymp}
Let $\mathcal{M}$ be a $3$ dimensional smooth connected manifold and let
$h$ be a metric on $\mathcal{M}$ such that $(\mathcal{M},h)$ is
complete. We say that $(\mathcal{M},h)$ is \textbf{asymptotically
  flat} of the class $H_{s,\delta}$ if $h\in H^s_{\rm loc}(\mathcal{M})$ and
there is a compact set $S\subset \mathcal{M}$ such that:
\begin{enumerate}
  \item[1.] There is a finite collection of charts $\{(U_i,\varphi_i)\}_{i=1}^N$
  which covers $\mathcal{M}\setminus S$;
  \item[2.] For each $i$, $\varphi_i^{-1}(U_i)=E_{r_i}:=\{x\in \setR^3:
|x|>r_i\}$ for
  some positive $r_i$;
  \item[3.]
  \label{item:3}
  The pull-back $(\varphi_{i \ast} h )_{ab}$ is uniformly equivalent to the
  Euclidean metric $\delta_{ab}$ on $E_{r_i}$ for each $i$;
  \item[4.]
  \label{item:4}
  For each $i$, $(\varphi_{i \ast} h )_{ab}-\delta_{ab}\in
H_{s,\delta}(E_{r_i})$.
\end{enumerate}
\end{defn}
 
The $H_{s,\delta}$-norm on the manifold $\mathcal{M}$ is defined as
follows.
Let $U_0\subset\mathcal{M}$ be an open set such that $S\subset U_0$ and
$\overline{U_0}\Subset\mathcal{M}$.
Let $\{\chi_0,\chi_i\}$ be a partition of unity subordinate to $\{U_0,U_i\}$,
then
\begin{equation}
\label{eq:norm-M}
 \left\|u \right\|_{H_{s,\delta}(\mathcal{M})}:= \left\|\chi_0u
\right\|_{H^{s}(U_0)}+\sum_{i=1}^N\left\|\varphi_i^\ast(\chi_i u)
\right\|_{H_{s,\delta}(\setR^3)}
\end{equation}
is the norm of the weighted fractional Sobolev space
$H_{s,\delta}(\mathcal{M})$.
For the definition of the norm $\left\|\chi_0u \right\|_{H^{s}(\Omega)}$ on the
manifold $\mathcal{M}$ see e.g.~\cite{Aubin98}.
Note that the norm (\ref{eq:norm-M}) depends on the partition of
unity, but different partitions of unity result in equivalent norms.
In the following we will omit the notation $\mathcal{M}$, that is, we
will write $ \|u\|_{H_{s,\delta}}$ instead of $\|u\|_{H_{s,\delta}(\mathcal{M})}$.

Since the principal symbol of the field equations (\ref{eq:eineul:1})
is characteristic in every direction (see e.g.~\cite{DeTurck}), it is
impossible to solve (\ref{eq:eineul:1}) in the present form.
We study these equations under the {\it harmonic gauge condition}
\begin{equation}
\label{eq:9}
 F^\mu=g^{\beta\gamma}\Gamma_{\beta\gamma}^\mu=0, 
\end{equation}
where $g^{\alpha\beta}$ is the inverse matrix of $g_{\alpha\beta}$. Then the
field equations (\ref{eq:eineul:1}) are equivalent to the {\it
reduced Einstein equations}
\begin{equation}
  \label{eq:publ-broken:18}
  g^{\mu\nu}\partial_\mu \partial_\nu g_{\alpha\beta}=
  H_{\alpha\beta}({g},\partial{g})-16\pi T_{\alpha\beta}+8\pi
  g^{\mu\nu}T_{\mu\nu}g_{\alpha\beta},
\end{equation}
where $H_{\alpha\beta}({g},\partial{g})$ is a quadratic function of
the semi--metric $g_{\alpha\beta}$ and its first order derivatives.
Since $g^{\mu\nu}$ has a Lorentzian signature,
(\ref{eq:publ-broken:18}) is a system of quasi-linear wave equations.
Taking into account the equation of state (\ref{eq:eineul:4}), the
normalization condition (\ref{eq:eineul:2}), and the Makino variable
(\ref{eq:intro:11}), then the system of wave equations (\ref{eq:publ-broken:18})
become
\begin{equation}
  \label{eq:publ-broken:11}
  g^{\mu\nu}\partial_\mu \partial_\nu g_{\alpha\beta}=
  H_{\alpha\beta}({g},\partial{g})-8\pi w^{\frac{2}{\gamma-1}}\left((1-Kw^2)
g_{\alpha\beta}+2(1+Kw^2)u_{\alpha}u_{\beta}\right).
\end{equation}
   
So the unknowns of the system (\ref{eq:publ-broken:11}) coupled with
the Euler equations (\ref{eq:eineul:3}) are the semi--metric
$g_{\alpha\beta}$, the velocity vector $u^\alpha$ and the Makino
variable $w$.
Note that even if $w$ is a smooth function, $w^{\frac{2}{\gamma-1}}$
might not be smooth in certain regions.
The initial data consist of the triple $(\mathcal{M},h_{ab},K_{ab})$,
where $\mathcal{M}$ is a space-like manifold,
$h_{ab}$ is a proper Riemannian metric on
$\mathcal{M}$ and $K_{ab}$ is its second fundamental form (extrinsic
curvature).

The semi--metric $g_{\alpha\beta}$ takes the following data on $M$:
\begin{equation}
\label{eq:2}
 \left\{\begin{array}{c}
g_{00}{\mid_{\mathcal{M}}}=-1 ,\quad g_{0a}{\mid_{\mathcal{M}}}=0,
\quad g_{ab}{\mid_{\mathcal{M}}}=h_{ab} \\
 -\frac{1}{2}\partial_0 g_{ab}{\mid_{\mathcal{M}}}=K_{ab},   \\
\end{array}\quad a,b=1,2,3.\right.
\end{equation}

The remaining initial data for
$\partial_0{g_{\alpha0}}_{\mid_{\mathcal{M}}}$ are
determined through the harmonic condition $F^\mu=0$. We compute them
by inserting the initial data (\ref{eq:2}) in the harmonic gauge
condition (\ref{eq:9}).
Since $\protect \partial_a
g_{00}{\mid_{\mathcal{M}}}=\partial_ag_{b0}{\mid_{\mathcal{M}}}=0$,
this results in the following expressions for
$\partial_0{g_{\alpha0}}_{\mid_{\mathcal{M}}}$:
\begin{equation}
\label{eq:2.7}
 \left\{\begin{array}{lll}
  \partial_0{ g_{00}}_{\mid_{\mathcal{M}}} & = &2 h^{ab}(K_{ab})      \\
\partial_0{g_{0c}}_{\mid_{\mathcal{M}}}& = &\frac{1}{2}
 \left(h^{ab}(\partial_a\protect h_{bc}-\partial_c
h_{ab})\right).
\end{array}
\right.
\end{equation}
In addition, the initial data of the velocity vector $u^\alpha$ and the
Makino variable $w$ are given on $\mathcal{M}$.
We denote the Minikowski metric by $\eta_{\alpha\beta}$.
\begin{thm}[Main result]
  \label{thm:main}
  Let $\frac{3}{2}<s<\frac{2}{\gamma-1}+\frac{1}{2}$ and
  $-\frac{3}{2}\leq \delta$.
  Assume $\mathcal{M}$ is asymptotically flat of class
  $H_{s+1,\delta}$, $K_{ab}\in H_{s,\delta+1}$,
  $\left(u^0-1,u^a,w\right){\big|_\mathcal{M}}\in H_{s+1,\delta+{1}}$,
  $w(0)\geq 0$ and $u^\alpha(0)$ is a timelike vector.
  Then there exists a positive $T$, a unique semi--metric
  $g_{\alpha\beta}$, a unit timelike vector $u^\alpha$ and $w$
  satisfying the reduced Einstein equations (\ref{eq:publ-broken:11})
  and the Euler equations (\ref{eq:eineul:3}) such that
  \begin{equation}
    \label{eq:2.13}
    ({g}_{\alpha\beta}(t)-\eta_{\alpha\beta})\in C([0,T],H_{s+1,\delta})\cap
    C^1([0,T],H_{s,\delta+1})
  \end{equation}
  and
  \begin{equation}
    \label{eq:2.14}
    \left(u^0-1,u^a,w\right)\in  C([0,T],H_{s+1,\delta+1})\cap
    C^1([0,T],H_{s,\delta+2}).
  \end{equation}
\end{thm}

\begin{rem}[On the differentiability]
  \label{rem:sec2-main-result:2}
Note that we have a lower and an upper bound of the differentiability index
$s$, however, in case $\frac{2}{\gamma-1}$ is an integer, then there is no upper
bound.
\end{rem}

A necessary and sufficient condition for the equivalence between the reduced
Einstein equations (\ref{eq:publ-broken:11}) and the field equations
(\ref{eq:eineul:1}) is that the geometric data $(h,K_{ab})$ satisfy the
constraint equations 
 \begin{equation}
    \label{eq:intro:6}
    \left\{
      \begin{array}{ccc}
        R(h) - K_{ab}K^{ab}+(h^{ab}K_{ab})^2 &=& 16\pi z\\
        {}^{(3)}\nabla_b K^{ab}-
        {}^{(3)}\nabla^b(h^{bc}K_{bc}) &=&-8\pi j^a.
      \end{array}\right.
  \end{equation}
  Here $R(h)=h^{ab}R_{ab}$ is the scalar curvature with respect to the
  metric $h_{ab}$.
  The right hand-side $(z,j^a)$ consists of the energy density and the
  momentum density respectively.

  Note that the harmonic coordinates  $F^\mu$ satisfy a
  homogeneous wave equation.
  That is why $F^\mu \equiv 0$, if $F^\mu{\mid{_\mathcal{M}}}=0$ and
  $\partial_0F^\mu{\mid{_\mathcal{M}}}=0$.
  The first condition follows from (\ref{eq:2.7}), and the second 
  holds if the reduced Einstein equations (\ref{eq:publ-broken:11})
  and the constraint equations (\ref{eq:intro:6}) are satisfied
  \cite{CHY}, \cite[\S18.8]{taylor97c}, \cite[\S
  10.2]{wald84:_gener_relat}.
  Although some of these references concern only the vacuum equations,
  since the proof relies on the Bianchi identities, it is valid for the
  Einstein-Euler equations, whose energy-momentum tensor
  $T^{\alpha\beta} $ is divergence free.

  Thus solving the constraint equations (\ref{eq:intro:6}) ensures
  that the solution of (\ref{eq:publ-broken:11}) satisfies the
  original system (\ref{eq:eineul:1}).
  However, before we solve the constraints, we need to treat the
  compatibility problem between the matter variables $(z,j^a)$ and the
  initial data for the velocity $u^\alpha$ and the Makino variable
  $w$.

  This problem can be described as follows: Let $\bar u^\alpha$ denote
  the projection of the velocity vector $u^\alpha$ on the initial
  manifold $\mathcal{M}$ and $n^\alpha$ the timelike unit normal
  vector to $\mathcal{M}$.
  The energy density $z$ is the double projection of $T_{\alpha\beta}$
  on $n^\alpha$ and the momentum density $j^a$ is once the projection
  of $T_{\alpha\beta}$ on $n^\alpha$ and once on $\mathcal{M}$.
  Applying these projections to the perfect fluid (\ref{eq:eineul:2})
  results in
  \begin{equation}
  \label{eq:intro:7}
  \left\{
    \begin{array}{ccc}
      z &=
&w^{\frac{2}{\gamma-1}}\left(1+(1+Kw^2)\right)h_{ab}\bar{u}^a\bar{u}^b\\
      j^a
      &=&w^{\frac{2}{\gamma-1}}\left(1+Kw^2\right)\bar{u}^a
\sqrt{1+h_{bc}\bar{u}^b\bar{u}^c}
    \end{array}\right..
\end{equation}

So the compatibility problem consists of solving (\ref{eq:intro:7})
for $w$, $\bar {u}^{a}$, when $z$ and $j^a$ are given.
This problem combined with a solution to the constraint equations
(\ref{eq:intro:6}) was solved in the $H_{s,\delta}$-spaces in
\cite[Theorem 2.6]{BK3}.
The conditions for this result are that
$\frac{3}{2}<s<\frac{2}{\gamma-1}+\frac{1}{2}$, where the metric
$h_{ab}-\delta_{ab}\in H_{s+1,\delta}$ with
$\delta\in(-\frac{3}{2},-\frac{1}{2})$, while for the mater variables
$(z,j^a)\in H_{s,\delta}$ and $\delta$ is just bounded below by
$-\frac{3}{2}$.
Note that for the hyperbolic equations we need one more degree of
regularity, so we need to require that $(z,j^a)\in H_{s+1,\delta+1}$.
But then the Makino variable (\ref{eq:intro:11}) causes that the upper
bound for $s$ becomes $\frac{2}{\gamma-1}-\frac{1}{2}$.
Given this restriction, we have by Theorem 2.5 of \cite{BK3} and
Proposition \ref{Moser} below that
$(w,u^0-1,u^a){\big|_{\mathcal{M}}}\in H_{s+1,\delta+1}$.

Thus relying on \cite{BK3}, we conclude that there is an initial data
set $(h_{ab},K_{ab})$ and $(w,u^\alpha)$ belonging to the
$H_{s,\delta}$-spaces that satisfies both the constraints
(\ref{eq:intro:6}) and the compatibility problem (\ref{eq:intro:7}).
The parameter $\gamma$ of these initial data, however, belongs to the
interval $(1,2)$.

\begin{cor}
  Under the assumptions of Theorem \ref{thm:main} and in addition
  under the assumption that the initial data $(h_{ab},K_{ab})$ and
  $(w,u^\alpha)$ satisfy the constraint equations (\ref{eq:intro:6})
  and compatibility problem (\ref{eq:intro:7}), there exists a
  positive $T$, a semi--metric $g$, a unit timelike vector $u^\alpha$
  and $w$ satisfying the Einstein (\ref{eq:eineul:1}) and the Euler
  equations (\ref{eq:eineul:3}) for $t\in[0,T]$.
  The regularity of $g$, $u^\alpha$ and $w$ are the same as in Theorem
  \ref{thm:main}.
\end{cor}

\begin{rem}[Existence, Uniqueness and Regularity]
  \label{rem:section-2:2}
  Existence, uniqueness and regularity of solutions of the Einstein
  equations (\ref{eq:eineul:1}) hold relative to the harmonic
  coordinate condition (\ref{eq:9}).
  Geometrical uniqueness requires usually one degree more of
  differentiability
  \cite{FMA% Fischer \& Marsden 1972, Communications in Mathematical Physics 28, 1
  }.
  Planchon and Rodnianski
  \cite{Planchon_Rodnianski_07% Planchon \& Rodnianski 2007
  }(see also \cite{CGP_08}) gave an argument for the vacuum case
  to get rid of this additionally regularity.
  For the Einstein--Euler system, and other matter fields, however,
  this problem remains still open.
\end{rem}

\section{Symmetric Hyperbolic Systems}
\label{sec:init-value-probl}
The main result is proved by transforming the coupled system
(\ref{eq:publ-broken:11}) and (\ref{eq:eineul:3}) into a symmetric
hyperbolic system.
We therefore recall its definition.
\begin{defn}[Symmetric hyperbolic system]
  \label{def:euler-rel:1}
  A first order quasi--linear $k\times k$ system is {\it symmetric
    hyperbolic system} in a region $G\subset \setR^k$, if it is of the
  form
  \begin{equation}
    \label{eq:publ-broken:7}
    L\lbrack   U \rbrack      =     
    A^\alpha (U)\partial_\alpha U  +
    B(U) = 0,
  \end{equation}
  where the matrices $A^\alpha(U) $ are symmetric and for every arbitrary
  $U\in G $, there exists a covector $\xi$ such that
  \begin{equation}
    \label{eq:publ-broken:8}
    \xi_\alpha   A^\alpha (U)
\end{equation} 
is  positive   definite.
The covectors $\xi_\alpha$ for which (\ref{eq:publ-broken:8}) is
positive definite, are called {\it timelike with respect to equation
}(\ref{eq:publ-broken:7}).
\end{defn}
If $\xi$  can be chosen to be the vector $(1,0,0,0)$, then 
 condition (\ref{eq:publ-broken:8}) implies that the matrix $A^0(U)$
is a positive definite matrix, and we may write system (\ref{eq:publ-broken:7})
in the form
\begin{equation}
\label{eq:symm}
 A^0(U)\partial_t U=A^a(U)\partial_a U+B(U).
\end{equation} 
\subsection{The Euler equations written as a symmetric hyperbolic
  system}
\label{sec:euler-equat-writt}
It is not obvious that the Euler equations written in the conservative
form $\nabla_{\alpha}T^{\alpha\beta}=0$ are symmetric hyperbolic.
In fact these equations have to be transformed in order to be
expressed in a symmetric hyperbolic form.
Rendall presented such a transformation of these equations in
\cite{Rendall92:-fluid}, however, its geometrical meaning is not
entirely clear and it might be difficult to generalize it to the non
time symmetric case.
Hence we will present a different hyperbolic reduction of the Euler
equations and discuss it in some details, for we have not seen it
anywhere in the literature.

The basic idea is to perform the standard \textit{fluid decomposition}
and then to modify the equation by adding, in an appropriate manner,
the normalization condition (\ref{eq:publ-broken:2}) which will be
considered as a constraint equation.
The fluid decomposition method consists of the projection of equation
$ \nabla_{\nu}T^{\nu\beta} = 0 $ onto $u^{\alpha}$ which leads to
\begin{math}
    u_\beta\nabla_{\nu}T^{\nu\beta} =  0
    \label{kap3.flp20}
  \end{math}, and the projection of these equations on rest pace
${\mathcal O}$ orthogonal to $u^\alpha$ of a fluid
which leads to
  \begin{math}
    P_{\alpha\beta}\nabla_{\nu}T^{\nu\beta} = 0 \end{math}, where
$ P_{\alpha\beta}= g_{\alpha\beta}+u_\alpha u_\beta$.
Inserting this decomposition into (\ref{eq:eineul:2}) results in a
system of the following form: 
\begin{subequations}
  \begin{eqnarray}
    \label{eq:euler-rel:10}
    u^{\nu}\nabla_{\nu}\epsilon + (\epsilon+p) \nabla_{\nu}u^{\nu}
    &=& 0; \\[0.2cm]
    \label{kap3.flz3}
    (\epsilon+p)   P_{\alpha\beta}u^{\nu}\nabla_{\nu}{u^\beta} +
P^\nu{}_\alpha
    \nabla_\nu p     &=&
    0.
  \end{eqnarray}
\end{subequations}

Note that we have beside the evolution equations
(\ref{eq:euler-rel:10}) and (\ref{kap3.flz3}) the following constraint
equation: $g_{\alpha\beta}u^{\alpha}u^{\beta}=-1$.
We will show in Subsection \ref{sec:cons-constr-equat} that this
constraint equation is conserved under the evolution equation.
In order to obtain a symmetric hyperbolic system we have to modify it
in the following way.
The normalization condition (\ref{eq:publ-broken:2}) gives that $
u_\beta u^\nu\nabla_\nu u^\beta = 0$, so we add $(\epsilon+p) u_\beta
u^\nu \nabla_\nu u^\beta = 0$ to equation (\ref{eq:euler-rel:10}) and
$ u_\alpha u_\beta u^\nu\nabla_\nu u^\beta = 0$ to (\ref{kap3.flz3}),
which results in
\begin{subequations}
  \begin{eqnarray}
    \label{eq:eineul:8}
    u^{\nu}\nabla_{\nu} \epsilon + (\epsilon+p)
P^\nu{}_\beta\nabla_\nu
u^\beta&=& 0 \\
    \label{eq:publ-broken:3} \Gamma_{\alpha\beta}
    u^{\nu}\nabla_{\nu}u^{\beta} +\frac{\sigma^2}{(\epsilon+p)}
P^\nu{}_\alpha
    \nabla_\nu \epsilon  &=& 0,
  \end{eqnarray}
\end{subequations}
where $\sigma:=\sqrt{\frac{\partial p}{\partial \epsilon }}$ is the speed of
sound and
\begin{equation*}
 \Gamma_{\alpha\beta}=
P_{\alpha\beta}+u_{\alpha}u_{\beta}= g_{\alpha\beta}+2u_{\alpha}u_{\beta}
\end{equation*}
is a reflection with respect to the rest subspace $\mathcal{O}$.
As mentioned above, we will introduce a new matter variable which is
given by (\ref{eq:intro:11}).
The idea  behind  is the following: The system
(\ref{eq:eineul:8}) and (\ref{eq:publ-broken:3}) is almost of
symmetric hyperbolic form, it is  symmetric if we multiply the
system by appropriate factors, for example, (\ref{eq:eineul:8}) by $
\frac{\partial p}{\partial \epsilon }=\sigma^2$ and
(\ref{eq:publ-broken:3}) by $(\epsilon+p)$.
However, doing so we will be faced with a system in which the
coefficients will either tend to zero or to infinity, as $\epsilon\to
0$.
Hence, it is impossible to represent this system in a non-degenerate
form using these multiplications.

The central point is now to introduce a new variable $w=M(\epsilon)$
which will regularize the equations even for $\epsilon=0$.
We do this by multiplying equation (\ref{eq:eineul:8}) by $\kappa^2
M'=\kappa^2\frac{\partial M}{\partial\epsilon}$.
This results in the following system which we have written in  matrix
form:
\begin{equation}
  \label{eq:Initial:1}
  \left(
    \begin{array}{c|cccc}
      \kappa^2u^\nu    &  & \kappa^{2}(\epsilon+p) M^{\prime}
      P^{\nu}{}_{\beta}  &     \\ \hline
      &  &                     &  &  \\
      \frac{\sigma^2}{(\epsilon+p) M^{\prime}}  P^{\nu}{}_{\alpha}
      &  & \Gamma_{\alpha\beta} u^\nu       &     \\
      &  &                     &  &  \\
    \end{array}
  \right)\nabla_\nu\left(
    \begin{array}{c}
      w \\
      u^\beta \\
    \end{array}
  \right)=\left(
    \begin{array}{c}
      0 \\
      0 \\
    \end{array}
  \right).
\end{equation}

In order to obtain symmetry we have to demand that
\begin{equation}
  \label{eq:Initial:6}
  M'=\frac{\sigma}{(\epsilon+p) \kappa},
\end{equation}
where ${\kappa} \gg0$ has been introduced in order to simplify the
expression for $w$.
If we choose $\kappa=\frac{2}{\gamma-1}
\frac{\sqrt{K\gamma}}{1+K\epsilon^{\gamma-1}}$, then
(\ref{eq:Initial:6}) holds and consequently the system
(\ref{eq:Initial:1}) is transferred to the symmetric system
\begin{equation}
  \label{eq:Initial:2} \left(
    \begin{array}{c|cccc}
      \kappa^2 u^\nu    &  & \sigma\kappa P^{\nu}{}_{\beta}  &     \\
\hline
      &  &                     &  &  \\
      \kappa\sigma   P^{\nu}{}_{\alpha}
      &  & \Gamma_{\alpha\beta} u^\nu       &     \\
      &  &                     &  &  \\
    \end{array}
  \right)\nabla_\nu\left(%
    \begin{array}{c}
      w \\
      u^\beta \\
    \end{array}%
  \right)=\left(%
    \begin{array}{c}
      0 \\
      0 \\
    \end{array}%
  \right).
\end{equation}

The covariant derivative $\nabla_{\nu}$ takes in local coordinates the
form $\nabla_{\nu} = \partial_{\nu} + \Gamma(g^{\gamma\delta},\partial
g_{\alpha\beta})$, where $\Gamma$ denotes the Christoffel symbols.
This expresses the fact that system
(\ref{eq:Initial:2}) is coupled to system
(\ref{eq:eineul:1}) for the gravitational field $g_{\alpha\beta}$.
In addition, from the definition of the Makino variable
(\ref{eq:intro:11}), we see that $\epsilon^{\gamma-1}=w^2$, so
$\kappa=\frac{2}{\gamma-1} \frac{\sqrt{K\gamma}}{1+Kw^2}$ and
$\sigma=\sqrt{\gamma K}w$.
Thus the fractional power of the equation of state (\ref{eq:eineul:4})
does not appear in the coefficients of the system
(\ref{eq:Initial:2}), and these coefficients are $C^\infty$ functions
of the scalar $w$, the four vector $u^\alpha$ and the gravitational
field $g_{\alpha\beta}$.

Now we want to show that $A^0$ of our system (\ref{eq:Initial:2}) is
indeed positive definite.
In order to do it we analyze the principal symbol of this system.
For each $\xi_\alpha\in T_x^*V$, the principal symbol is a linear map
from $\mathbb{R}\times E_x$ to $\mathbb{R}\times F_x$, where $E_x$ is
a fiber in $T_xV$ and $F_x$ is a fiber in the cotangent space
$T_x^*V$.
Since in local coordinates $\nabla_{\nu} =\partial_{\nu} +
\Gamma(g^{\gamma\delta},\partial g_{\alpha\beta})$, the principal
symbol of system (\ref{eq:Initial:2}) is
\begin{equation}
  \label{eq:Initial:4}\xi_\nu A^\nu= \left(
    \begin{array}{c|cccc}
      \kappa^2    (u^\nu \xi_\nu)   &  & \sigma\kappa
      P^{\nu}{}_{\beta}\xi_\nu  &     \\ \hline
      &  &                     &  &  \\
      \sigma\kappa   P^{\nu}{}_{\alpha}\xi_\nu
      &  & (u^\nu\xi_\nu)\Gamma_{\alpha\beta}        &     \\
      &  &                     &  &  \\
    \end{array}
  \right).
\end{equation}
The characteristics are the set of covectors $\xi$ for which $(\xi_\nu
A^\nu)$ is not an isomorphism.
Hence the characteristics are the zeros of $Q(\xi):=\det(\xi_\nu
A^\nu)$.
The geometric advantages of the fluid decomposition are the following.
The operators in the blocks of the matrix (\ref{eq:Initial:4}) are the
projection on the rest hyperplane $\mathcal{O}$, $P^{\nu}{}_{\alpha}$,
and the reflection with respect to the same hyperplane, $
\Gamma_{\alpha\beta}$.
Therefore, the following relations hold:
\begin{displaymath}
  \Gamma^{\alpha\gamma}\Gamma_{\gamma\beta}=\delta_\beta{}^\alpha,
  \qquad \Gamma^{\alpha\gamma}P_\gamma{}^\nu=P^{\alpha\nu}\qquad
  \text{and}\qquad P_\beta{}^\alpha P_{\alpha}{}^\nu=P^{\nu}{}_\beta,
\end{displaymath}
which yields
\begin{equation}
  \label{eq:Initial:5}\left(
    \begin{array}{c|cccc}
      1   &  & 0  &     \\ \hline
      &  &                     &  &  \\
      0
      &  &\Gamma^{\alpha\gamma}         &     \\
      &  &                     &  &  \\
    \end{array}
  \right)
  \left(\xi_\nu A^\nu\right)= \left(
    \begin{array}{c|cccc}
      \kappa^2          (u^\nu \xi_\nu)   &  & \sigma\kappa
      P^{\nu}{}_{\beta}\xi_\nu  &     \\ \hline
      &  &                     &  &  \\
      \sigma\kappa   P^{\alpha\nu}{}\xi_\nu &
      &(u^\nu\xi_\nu)\left(\delta^{\alpha}_{\beta} \right)       &
\\
      &  &                     &  &  \\
    \end{array}
  \right).
\end{equation}

It is now fairly easy to calculate the determinate of the right hand
side of (\ref{eq:Initial:5}) and we have
\begin{displaymath}
  \det\left(
    \begin{array}{c|cccc}
      \kappa^2      (u^\nu \xi_\nu)   &  & \sigma\kappa
      P^{\nu}{}_{\beta}\xi_\nu  &     \\ \hline
      &  &                     &  &  \\
      \sigma\kappa   P^{\alpha\nu}{}\xi_\nu &
      &(u^\nu\xi_\nu)\left(\delta^{\alpha}_{\beta} \right)       &
\\
      &  &                     &  &  \\
    \end{array}
  \right)
  =\kappa^2(u^\nu\xi_\nu)^3\left((u^\nu\xi_\nu)^2 -\sigma^2
    P^{\alpha\nu}\xi_\nu P_\alpha ^\nu\xi_\nu\right).
\end{displaymath}
Since $P_\beta^\alpha$ is a projection,
\begin{equation*}
  P^{\alpha\nu}\xi_\nu P_\alpha ^\nu\xi_\nu= g^{\nu\beta}\xi_\nu
  P^{\alpha}_\beta P_\alpha ^\nu\xi_\nu
  =g^{\nu\beta}\xi_\nu P^{\nu}{}_\beta \xi_\nu
  =P^{\nu}{}_\beta\xi_\nu  \xi^\beta,
\end{equation*}
and since $\Gamma_\beta^\gamma$ is a
reflection,
\begin{equation*}
  \det\left(
    \begin{array}{c|cccc}
      1   &  & 0  &     \\ \hline
      0
      &  &\Gamma^{\alpha\gamma}         &     \\
    \end{array}
  \right) =\det\left(g^{\alpha\beta}\Gamma_\beta^\gamma\right)=
 -\left({\rm
      det}\left(g_{\alpha\beta}\right)\right)^{-1}.
\end{equation*}
Consequently,
\begin{equation}
  \label{eq:Initial:11}
  Q(\xi ):=  \det (\xi_\nu A^\nu  ) =-\kappa^{2}\det(g_{\alpha\beta})
  ( u^\nu\xi_\nu)^3
  \left\{ ( u^\nu\xi_\nu)^2 - \sigma^2
    P^{\alpha}{}_{\beta}\xi_\alpha \xi^\beta \right\}
\end{equation}
and therefore the characteristic covectors are given by two simple
equations:
\begin{eqnarray}
  \label{eq:publ-broken:39}
  \xi_\nu u^\nu & = &0 ;\\
  \label{eq:publ-broken:40} (\xi_\nu u^\nu)^2 - \sigma^2
  P^{\alpha}{}_{\beta}\xi_\alpha \xi^\beta & = &0.
\end{eqnarray}

\begin{rem}
  \label{rem:euler-rel:7}
  The characteristics conormal cone is a union of two hypersurfaces in
  $T_x^*V$.
  One of these hypersurfaces is given by the condition
  (\ref{eq:publ-broken:39}) and it is a three dimensional hyperplane
  $\mathcal{O}$ with the normal $u^\alpha$.
  The other hypersurface is given by the condition
  (\ref{eq:publ-broken:40}) and forms a three dimensional cone, the so
  called, {\it sound cone}.
\end{rem}
Let us now consider the timelike vector $u_{\nu }$ and  insert the
covector  $-u_\nu$ into the principal symbol
(\ref{eq:Initial:4}), then
\begin{equation*}
  -u_\nu A^\nu=\left(
    \begin{array}{c|cccc}
      \kappa^2 &  & 0         &     \\
      \hline
      &  &           &  &  \\
      0 &  & \Gamma_{\alpha\beta} &     \\
      &  &           &  &  \\
    \end{array}
  \right) \label{kap3.mk3}
\end{equation*}
is positive definite.
Indeed, $\Gamma_{\alpha\beta}$ is a reflection with respect to a
hyperplane having a timelike normal.
Hence, $-u_\nu$ is a timelike covector with respect to the
hydrodynamical equations (\ref{eq:Initial:2}).
Herewith, we have showed relatively elegant and elementary that the
relativistic hydrodynamical equations are symmetric--hyperbolic.
We want now to show that the covector $t_{\alpha}=(1,0,0,0)$ is also
timelike with respect to the system (\ref{eq:Initial:2}).
Since $P^\alpha{}_{\beta}u_\alpha=0$, the covector $-u_\nu$ belongs to
the sound cone
\begin{equation}
  \label{eq:Initial:9}
  (\xi_\nu u^\nu)^2 - \sigma^2P^{\alpha}{}_{\beta}\xi_\alpha
  \xi^\beta>0.
\end{equation}

Inserting $t_\nu=(1,0,0,0)$ the right hand side of
(\ref{eq:Initial:9}) yields
\begin{equation}
  \label{eq:Initial:10}
  (u^0)^2(1-\sigma^2)-\sigma^2 g^{00}.
\end{equation}
Since the sound velocity is always less than the light speed, that is
$\sigma^2=\frac{\partial p}{\partial \epsilon}<c^2=1$, we conclude
from (\ref{eq:Initial:10}) that $t_\nu$ also belongs to the sound cone
(\ref{eq:Initial:9}).
Hence, the vector $-u_\nu$ can be continuously deformed to $t_\nu$
while condition (\ref{eq:Initial:9}) holds along the deformation path.
Consequently, the determinant of (\ref{eq:Initial:11}) remains
positive under this process and hence $t_{\nu}A^\nu=A^{0}$ is also
positive definite.
Thus we have proved.

\begin{thm}
\label{thm:2}
  Let $\epsilon$ be non-negative density function, then the Euler
  system (\ref{eq:eineul:3}) coupled with the equation of state
  (\ref{eq:eineul:4}) can be written as a symmetric
  hyperbolic system of the form (\ref{eq:symm}), and where
  $A^0$ is a positive definite.
\end{thm}
 
\subsubsection{Conservation of the unit length vsctor of the fluid}
\label{sec:cons-constr-equat}

\begin{prop}
The  constraint condition $g_{\alpha\beta} u^\alpha
u^\beta = -1$ is conserved along stream lines $u^\alpha$.
\end{prop}
\begin{proof}
Let $k(t)$ be a curve such that $k'(t)=u^\alpha$ and set $Z(t)=(u\circ
k)_\beta(u\circ k)^\beta$. In order to establish the conservation of the
constraint condition it suffices to establish the following relation
\begin{equation*}
  \frac{d}{dt}Z(t)=2 u_\beta \nabla_{k'(t)} u^\beta
  =2u^\nu u_\beta\nabla_\nu u^\beta= 0.
  \label{kap3.43A}
\end{equation*}
Multiplying the  four last rows of the Euler system
(\ref{eq:Initial:2}) by $u^\alpha$ and recalling that
$P^{\nu}{}_\alpha$ is the projection on the rest space $\mathcal{O}$
orthogonal to $u^\alpha$, we have
\begin{eqnarray*}
  0 &=& u^\alpha\left(\Gamma_{\alpha\beta}u^\nu\nabla_\nu u^\beta +
    \kappa\sigma P^{\nu}{}_\alpha\nabla_\nu w\right)\\
  &=& u^\alpha P_{\alpha\beta}u^\nu\nabla_\nu u^\beta - u^\nu
  u_\beta\nabla_\nu u^\beta
  + \kappa\sigma u^\alpha P^{\nu}{}_\alpha\nabla_\nu w\\ &=&
  - u^\nu u_\beta\nabla_\nu u^\beta.
\end{eqnarray*}
Therefore, if $g_{\alpha\beta} u^\alpha u^\beta = -1$ on the initial manifold,
then it holds along the stream lines $u^\alpha$.
\end{proof}
\subsection{The coupled hyperbolic system}
\label{sec:evolution-equations}
In this section we will transform the coupled system
(\ref{eq:publ-broken:11}) and (\ref{eq:Initial:2}) into a symmetric
hyperbolic system.
We will pay attention to the fact that the system will be in a form in
which we can apply the energy estimates of \cite{Ka1}.
That allows us to obtain the same regularity for the gravitational
fields as Hughes, Kato and Marsden \cite{hughes76:_well} got for the
Einstein vacuum equations.
Note that our system is slightly different from the symmetric
hyperbolic system obtained by Fisher and
Marsden~\cite{FMA% Fischer \& Marsden 1972, Communications in
%
% Mathematical Physics 28, 1
}, since our system contains a constant matrix $\mathcal{C}^a$ as
given by (\ref{eq:sec3-hyperbolic:1}).

We consider a spacetime $(V,g_{\alpha\beta})$ of the type $\setR\times
\mathcal{M}$, where $\mathcal{M}$ is a Riemannian manifold,
and we denote local coordinates by   $(t,x^a)$.
Set
\begin{equation*}
h_{\alpha\beta\gamma}=
\partial_\gamma g_{\alpha\beta},
\end{equation*}
 then the reduced Einstein
equations  (\ref{eq:publ-broken:11}) takes the form
\begin{equation}
  \label{eq:2.3}
  \begin{array}{lll}
    \partial_t  g_{\alpha\beta}&=& h_{\alpha\beta  0} \\
    -g^{00}\partial_t  h_{\alpha\beta  0} &=&
\left\{2g^{0a}\partial_a h_{\alpha\beta   0} + g^{ab}\partial_a
h_{\alpha\beta b} + H_{\alpha\beta}(g,\partial g)\right. \\ &-& \left.8\pi
w^{\frac{2}{\gamma-1}}\left((1-Kw^2)
g_{\alpha\beta}+2(1+Kw^2)u_{\alpha}u_{\beta}\right)\right\}
 \\
   g^{ab}\partial_t  h_{\alpha\beta  a}
&=& g^{ab}\partial_a h_{\alpha\beta 0}
\end{array}.
\end{equation}

In order to apply the energy estimates of \cite{Ka1}, we need that the
coefficients of $ \partial_t h_{\alpha\beta0}$ will be independent of
$t$.
This is because of  the specific form of the inner-product in
$H_{s,\delta}$ spaces which takes into account the matrix $A^0$ of the
system (\ref{eq:publ-broken:7}).
In Section \ref{sec:energy-estimates} we will further clarify this
issue.
Therefore we divide the second row by $ -g^{00}$ and in order to
preserve the symmetry of the system, we also multiply the third row by
$(-g^{00})^{-1} $.
Thus the system of wave equations (\ref{eq:publ-broken:11}) are equivalent to
the system
\begin{equation}
  \label{eq:2.4}
  \begin{array}{lll}
    \partial_t  g_{\alpha\beta} &=& h_{\alpha\beta  0} \\
    \partial_t  h_{\alpha\beta  0} &=& {(-g^{00})^{-1}}\left\{
    2g^{0a}\partial_a h_{\alpha\beta   0} + g^{ab}\partial_a
h_{\alpha\beta b}+ H_{\alpha\beta}(g,\partial g) \right.\\ &-&
\left.
8\pi
w^{\frac{2}{\gamma-1}}\left((1-Kw^2)
g_{\alpha\beta}+2(1+Kw^2)u_{\alpha}u_{\beta}\right)\right\}
 \\
   (-g^{00})^{-1} g^{ab}\partial_t  h_{\alpha\beta  a}
&=&(-g^{00})^{-1}g^{ab}\partial_a h_{\alpha\beta 0}
\end{array}.
\end{equation}

To shorten and simplify the notation, we introduce the auxiliary variables
\begin{equation*}
 (v,\partial_t v,\partial_x
v)=(g_{\alpha\beta}-\eta_{\alpha\beta}, \partial_t g_{\alpha\beta}, \partial_x
g_{\alpha\beta}),
\end{equation*}
where $\eta_{\alpha\beta} $ denotes the Minkowski metric and
$\partial_x$ denotes the set of all spatial derivatives.
We also set $(e_0)^\alpha=(1,0,0,0)$ and $W=(w, u^\alpha-e_0^\alpha)$
stands for the Makino and the fluid variables.
Finally,
\begin{equation*}
 U=(v,\partial_t v,\partial_x v,W)
\end{equation*}
represents the unknowns of the coupled system.

We write the matrices in a block form,  $A=({\bf a}_{ij})$, the $k\times k$
identity matrix is denoted by ${\bf e}_k$ and ${\bf 0}_{m\times n}$ is the zero
matrix.

The coupled system (\ref{eq:2.4}) and
(\ref{eq:Initial:2}) can be written in the form of (\ref{eq:publ-broken:7}),
where $\mathcal{A}^\alpha$ are $55\times 55 $ symmetric matrices which depend
only on $v$ and $W$. We shall describe now the structure of these matrices:
\begin{equation}
\label{eq:3.7}
\mathcal{A}^0(v,W)=\left(\begin{array}{llll}
{\bf e}_{10} & {\bf 0}_{10\times 10} &{\bf 0}_{10\times 30} &{\bf 0}_{10\times
5}  \\
{\bf 0}_{10\times 10} & {\bf e}_{10} & {\bf 0}_{10\times 30}&{\bf 0}_{10\times
5}\\
{\bf 0}_{30\times 10}   & {\bf 0}_{30\times 10} & {\bf a}_{33}^0 &{\bf
0}_{30\times 5}
\\
{\bf 0}_{5\times 10}   & {\bf 0}_{5\times 10} & {\bf 0}_{5\times 30} &{\bf
a}_{44}^0
\end{array} \right),
\end{equation}
where
\begin{equation*}
 {\bf a}_{33}^0=\dfrac{1}{-g^{00}}\left(
\begin{array}{ccc} {g}^{11}\e_{10} &{g}^{12}\e_{10} & {g}^{13}\e_{10}\\
{g}^{21}\e_{10} & {g}^{22}\e_{10} & {g}^{23}\e_{10} \\
{g}^{31}\e_{10} & {g}^{32}\e_{10} & {g}^{33}\e_{10}
\end{array}\right),
\end{equation*}
and ${\bf a}_{44}^0={\bf a}_{44}^0(g_{\alpha\beta},w,u^\alpha)$ is
given by
(\ref{eq:Initial:2}) when $\nu=0$.
From (\ref{eq:2.4}) we see that the coefficients of $\partial_a U$, $a=1,2,3$,
have the form
\begin{equation*}
\left(\begin{array}{lll}
{\bf 0}_{10\times 10} & {\bf 0}_{10\times 40} &{\bf 0}_{5\times 5}  \\
{\bf 0}_{40\times 10} & {\bf a}_{22}^a & {\bf 0}_{40\times 5}\\
{\bf 0}_{5\times 10}   & {\bf 0}_{5\times 40} & {\bf a}_{33}^a
\end{array} \right),
\end{equation*}
where ${\bf a}_{33}^a={\bf a}_{33}^a(g_{\alpha\beta},w,u^\alpha)$ is from the
system (\ref{eq:Initial:2}) of the fluid and
\begin{equation}
\label{eq:3.8}
{\bf a}_{22}^a(g_{\alpha\beta})=\dfrac{1}{g^{00}}
\left(
    \begin{array}{c|ccc}
     2g^{a0}{\bf e}_{10} & g^{a1}{\bf e}_{10} &g^{a2}{\bf e}_{10} & g^{a3}{\bf
e}_{10}  \\
      \hline
g^{a1}{\bf e}_{10}        &           &  &  \\
g^{a2}{\bf e}_{10}       &  \ &    {\Large {\bf 0}_{30\times30} } \\
g^{a3}{\bf e}_{10}       &           &  &  \\
    \end{array}
  \right).
\end{equation}

It is essential to demand that ${\bf a}_{22}^a(g_{\alpha\beta})\in
H_{s,\delta}$, whenever $g_{\alpha\beta}-\eta_{\alpha\beta}\in
H_{s,\delta}$.
Obviously, this does not hold for the matrix in (\ref{eq:3.8}).
Therefore we need to modify these matrices by a constant matrix
\begin{equation*}
{{\bf c}_{22}^a}=\left(
    \begin{array}{c|ccc}
      {\bf 0}_{10\times10} & \delta^{a1}{\bf e}_{10} &
\delta^{a2}{\bf e}_{10}        &    \delta^{a3}{\bf e}_{10} \\
      \hline
\delta^{a1}{\bf e}_{10}        &           &  &  \\
\delta^{a2}{\bf e}_{10}       &  \ &    {\Large {\bf 0}_{30\times30}}  \\
\delta^{a3}{\bf e}_{10}       &           &  &  \\
    \end{array}
  \right),
\end{equation*}
then $\left({\bf a}_{22}^a-{\bf c}_{22}^a\right)(v)\in H_{s,\delta}$ whenever
$v\in H_{s,\delta}$. So we set
\begin{equation}
\label{eq:3.9}
\mathcal{A}^a(v,W)=\left(\begin{array}{lll}
{\bf 0}_{10\times 10} & {\bf 0}_{10\times 40} &{\bf 0}_{5\times 5}  \\
{\bf 0}_{40\times 10} & {\bf a}_{22}^a - {\bf c}_{22}^a & {\bf 0}_{40\times 5}\\
{\bf 0}_{5\times 10}   & {\bf 0}_{5\times 40} & {\bf a}_{33}^a
\end{array} \right)
\end{equation}
and a constant matrix
\begin{equation}
\label{eq:sec3-hyperbolic:1}
\mathcal{C}^a=\left(\begin{array}{lll}
{\bf 0}_{10\times 10} & {\bf 0}_{10\times 40} &{\bf 0}_{5\times 5}  \\
{\bf 0}_{40\times 10} & {\bf c}_{22}^a & {\bf 0}_{40\times 5}\\
{\bf 0}_{5\times 10}   & {\bf 0}_{5\times 40}  &{\bf 0}_{5\times 5}
\end{array} \right).
\end{equation}

We turn now to the lower order terms.
The presence of the fractional power $w^{2/(\gamma-1)}$ in
(\ref{eq:2.4}) causes substantial technical difficulties.
We set
\begin{equation}
\label{eq:3.10}
f(v,W):=- \dfrac{8\pi
w^{\frac{2}{\gamma-1}}}{g^{00}}\left((1-Kw^2)
g_{\alpha\beta}+2(1+Kw^2)u_{\alpha}u_{\beta}\right),
\end{equation} 
then we can write $B(U)$ in the form
\begin{equation*}
 B(U)=\mathcal{B}(U)(v,\partial_t v,\partial_x v)^T +\mathcal{F}(v,W),
\end{equation*}
where $\mathcal{F}(v,W)=(0,f(v,W),0,0)^T$ and 
\begin{equation}
\label{eq:3.6}
\mathcal{B}(U)=\left(\begin{array}{lllll}
{\bf 0}_{10\times 10} & {\bf e}_{10} &{\bf 0}_{10\times 10}  & {\bf
0}_{10\times 10}  & {\bf 0}_{10\times10}     \\
{\bf 0}_{10\times 10} & {\bf b}_{22} & {\bf b}_{23} & {\bf b}_{24}  & {\bf
b}_{25}\\
{\bf 0}_{30\times 10}   & {\bf 0}_{30\times 10} & {\bf 0}_{30\times 10} & {\bf
0}_{30\times10} & {\bf 0}_{30\times 10} 
\\
{\bf 0}_{5\times 10}   & {\bf b}_{42} & {\bf b}_{43} & {\bf b}_{44} &
{\bf b}_{45}\end{array} \right).
\end{equation}
The block ${\bf b}_{2j}$, $j=2,3,4,5$, appear from the quadratic terms in
(\ref{eq:publ-broken:18}):
\begin{equation*}
 H_{\alpha\beta}(g,\partial g)=
C_{\alpha\beta\gamma\delta\rho\sigma}^{\epsilon\zeta\eta\kappa\lambda\mu}h_{
\epsilon\zeta\eta}h_{\kappa\lambda\mu}g^{\gamma\delta}g^{\rho\sigma},
\end{equation*}
where 
$C_{\gamma\delta\alpha\beta\rho\sigma}^{\epsilon\zeta\eta\kappa\lambda\mu}$ 
are a combination of Kronecker deltas with
integer coefficients. 
Thus
\begin{equation*}
 {\bf b}_{2j}=(-g^{00})^{-1}
C_{\alpha\beta\gamma\delta\rho\sigma
}^{\epsilon\zeta\eta\kappa\lambda\mu}h_{\epsilon\zeta\eta}g^{\gamma\delta}g^{
\rho\sigma}, \ \mu=j-2.
\end{equation*}
The block ${\bf b}_{4j}$, $j=2,3,4,5$, appear from the multiplication
of the reflection $\Gamma_{\alpha\beta}$ and $ u^\nu$ in
(\ref{eq:Initial:2}) with the Christoffel symbols.
So its coefficients consist of multiplications of $g_{\alpha\beta}$,
$g^{\alpha\beta}$ and $ u^\nu$.

In summary, we can write the coupled systems (\ref{eq:2.4}) and
(\ref{eq:Initial:2}) as a symmetric hyperbolic system
\begin{equation}
 \label{eq:3.4}
\mathcal{A}^{0}(v,W)\partial_t
U=(\left(\mathcal{A}^a(v,W)+\mathcal{C}^a\right)\partial_a U
+\mathcal{B}(U)\begin{pmatrix}v\\ \partial_t v\\ \partial_x
v\end{pmatrix}+\mathcal{F}(v,W),
\end{equation} 
where $ \mathcal{A}^{0}(U)$ is positive definite in the neighborhood of the
initial data (\ref{eq:2}),
$\mathcal{A}^{0}(0)-{\bf e}_{55}=\mathcal{A}^{a}(0)= 0$  and
$\mathcal{C}^a$ is a constant symmetric matrix.

\section{The $H_{s,\delta}$ spaces and their properties}
\label{sec:properties}

The definition of the weighted Sobolev spaces of fractional order
$H_{s,\delta}$, Definition \ref{def:weighted:3}, is due to Trieble
\cite{triebel76:_spaces_kudrj2}. Here we quote the propositions and
properties which are needed for the proof of the main result. For
their proofs see \cite{BK3, maxwell06:_rough_einst,
  triebel76:_spaces_kudrj2}.

We start with some notations.
\begin{itemize}
  \item Let $\{\psi_j\}$ be the sequence of functions in Definition
  \ref{def:weighted:3}.  For any positive $\gamma$ we set
  \begin{equation}
    \label{eq:const:14}
   \|u\|_{H_{s,\delta,\gamma}}^2 = \sum_{j=0}^\infty 2^{(
        \frac{3}{2} + \delta)2j} \| (\psi_j^\gamma u)_{2^j}
      \|_{H^{s}}^{2},
  \end{equation}
  and we will use the convention $ \|u\|_{H_{s,\delta,1}}=
  \|u\|_{H_{s,\delta}}$.  The subscripts $2^j$ mean a scaling by
  $2^j$, that is, $ (\psi_j^\gamma u)_{2^j}(x)=(\psi_j^\gamma u)(2^j
  x)$.

  \item For a non-negative integer $m$ and $\beta\in \mathbb{R}$, the
  space $C_\beta^m$ is the set of all functions having continuous
  partial derivatives up to order $m$ and such that the norm
  \begin{equation}
    \label{eq:3.1}
    \|u\|_{C^m_\beta}=\sum_{|\alpha|\leq
      m}\sup_{\mathbb{R}^3}\left((1+|x|)^{\beta+|\alpha|} |\pa^\alpha
      u(x)|\right)
  \end{equation}
 is finite.

  \item We will use the notation $A\lesssim B$ to denote an inequality
  $A\leq C B$ where the positive constant $C$ does not depend on the
  parameters in question.
\end{itemize}

We recall that $\{\psi_j\}$ are cutoff functions, hence $\psi_j^\beta
\in C_0^\infty(\mathbb{R}^3)$ for any positive $\beta$. Furthermore,
for a given $\alpha$, there are two constants $C_1(\beta,\alpha)$ and
$C_2(\beta,\alpha)$ such that
\begin{equation*}
  C_1(\beta,\alpha)|\partial^\alpha \psi_j(x)|\leq
  |\partial^\alpha \psi_j^\beta(x)| 
  \leq C_2(\beta,\alpha)|\partial^\alpha \psi_j(x)|
\end{equation*}
and these inequalities are independent of $j$. 
Therefore Proposition \ref{equivalence:1} below is a consequence of
\cite[Theorem 1]{triebel76:_spaces_kudrj2}.

\begin{prop}
 \label{equivalence:1}
For any positive $\gamma$, there are two positive constants $c_0(\gamma)$ and
$c_1(\gamma)$ such that
\begin{equation*}
c_0(\gamma)  \|u\|_{H_{s,\delta}}\leq \|u\|_{H_{s,\delta,\gamma}}\leq
c_1(\gamma)  \|u\|_{H_{s,\delta}}.
  \end{equation*}
\end{prop}

\begin{prop}
 \label{equivalence}  For any nonnegative integer $m$, positive $\gamma$
and $\delta$ it holds
\begin{equation*}
  \|u\|_{H_{m,\delta,\gamma}}^2 \lesssim \|u\|_{m,\delta}^2
\lesssim   \|u\|_{H_{m,\delta,\gamma}}^2,
\end{equation*}
where  $\|u\|_{m,\delta}$ is defined by (\ref{eq:intro:5}).
\end{prop}

\begin{prop}
 \label{monoton}
If $s_1\leq s_2$ and $\delta_1\leq\delta_2$, then
\begin{equation*}
 \|u\|_{H_{s_1,\delta_1}}\leq  \|u\|_{H_{s_2,\delta_2}}.
\end{equation*}
\end{prop}

\begin{prop}
 \label{embedding1}
If $u\in H_{s,\delta}$, then
\begin{equation*}
 \|\partial_i u\|_{H_{s-1,\delta+1}}\leq \| u\|_{H_{s,\delta}}.
\end{equation*}
\end{prop}

\begin{prop}
 \label{Algebra}  Let $s_1,s_2\geq s$,
$s_1+s_2>s+\frac{3}{2}$, $s_1+s_2\geq0$  and
    $\delta_1+\delta_2\geq\delta-\frac{3}{2}$. If $u\in H_{s_1,\delta_1}$ and
$v\in H_{s_2,\delta_2}$, then 
\begin{equation}
    \label{eq:elliptic_part:12}
\left\|uv\right\|_{H_{s,\delta}}\lesssim \left\|u\right\|_{H_{s_1,\delta_1}}
\left\|v\right\|_{H_{s_2,\delta_2}}.
    \end{equation}
\end{prop}

\begin{rem}
\label{rem:2}
If for a fixed constant $c_0$,  $u-c_0\in H_{s_1,\delta_1}$ and  $v\in
H_{s_2,\delta_2}$,  then  we can apply the multiplication property
(\ref{eq:elliptic_part:12}) to $(u-c_0)v$ and obtain
\begin{equation*}
 \left\| uv\right\|_{H_{s,\delta}}\lesssim \left(\left\|
u-c_0\right\|_{H_{s_1,\delta_1}} +|c_0|\right)\left\|
v\right\|_{H_{s_2,\delta_2}}.
\end{equation*} 
\end{rem}

\begin{prop}
  \label{Kateb}
  Let $u\in H_{s,\delta}\cap L^\infty$, $1<\beta$, $
  0<s<\beta +\frac{1}{2}$ and $\delta\in \mathbb{R}$, then
  \begin{equation*}
     \||u|^\beta\|_{H_{s,\delta}}\leq
    C(\|u\|_{L^\infty})
    \|u\|_{H_{s,\delta}}.
  \end{equation*}
\end{prop}

\begin{prop}
 \label{Embedding}
  If $s>\frac{3}{2} + m$
and $\delta+\frac 3  2\geq\beta$, then
\begin{equation}
\label{eq:em:2} \| {u}\|_{C^m_\beta}\lesssim \|{u}\|_{H_{s,\delta}},
\end{equation}
where $\| {u}\|_{C^m_\beta}$ is given by (\ref{eq:3.1}).
\end{prop}

\begin{prop}
 \label{Moser}
 Let $F:\setR^m\to\setR^l$ be a $C^{N+1}$-function such that $F(0)=0$ and where
 $N\geq [s]+1$. Then there is a constant $C$ such that for any $u\in
 H_{s,\delta}$
\begin{equation}
\label{eq:14}
\|F(u)\|_{H_{s,\delta}}\leq
C\|F\|_{C^{N+1}}\left(1+\|u\|_{L^\infty}^N\right)
\|u\|_{H_{s,\delta}}.
\end{equation}
\end{prop}

\begin{prop}
 \label{Density}\quad
\begin{enumerate}
\item[(a)] The class $C_0^\infty(\mathbb{R}^3)$ is dense in
$H_{s,\delta}$.
\item[(b)] Given $u\in H_{s,\delta}$, $s'>s\geq 0$ and $\delta'\geq\delta$.
Then
for $\rho>0$  there is $u_\rho\in C_0^\infty(\mathbb{R}^3)$ and  a
positive constant $C(\rho)$ such that
\begin{equation*}
\|u_\rho -u\|_{H_{s,\delta}}\leq \rho\quad \text{and}\quad
\|u_\rho\|_{H_{s',\delta'}}\leq C(\rho) \|u\|_{H_{s,\delta}}.
\end{equation*}
\end{enumerate}
\end{prop}

\subsection{Product spaces.}

The unknown of system (\ref{eq:3.4}) is a vector valued function 
\begin{equation*}
 U=\left(v,\partial_t v,\partial_x v,W\right),
\end{equation*}
where $v=g_{\alpha\beta}-\eta_{\alpha\beta}$ stands for the field variables and
$W=(w,u^\alpha-e_0^\alpha)$ stands for the fluid variables. We consider it in
the space
\begin{equation}
\label{eq:5.7}
 X_{s,\delta}:=H_{s,\delta}\times H_{s,\delta+1}\times H_{s,\delta+1}\times
H_{s+1,\delta+1},
\end{equation}
with the  norm (see (\ref{eq:const:14}))
\begin{equation}
\begin{split}
 \left\|U\right\|_{X_{s,\delta}}^2 &
=\left\|v\right\|_{H_{s,\delta,2}}^2 +
\left\|\partial_t v\right\|_{H_{s,\delta+1,2}}^2 +
\left\|\partial_x v\right\|_{H_{s,\delta+1,2}}^2 + 
\left\|W\right\|_{H_{s+1,\delta+1,2}}^2.
\end{split}
\end{equation}

\begin{rem}
\label{rem:1}
 Note that if $U=\left(v,\partial_t v,\partial_x v,W\right)\in X_{s,\delta}$,
then  $v\in H_{s+1,\delta}$. Because $U\in X_{s,\delta}$ implies that
$v\in H_{s,\delta}$
and $\partial_x v\in H_{s,\delta+1}$, so by the
integral representation of the norm $H_{s,\delta}$ (see  \cite[\S2)]{BK3}), 
we obtain that
\begin{equation*}
 \left\| v \right\|_{H_{s+1,\delta}}\lesssim \left(\left\|
v\right\|_{H_{s,\delta}}+\left\| \partial_x v \right\|_{H_{s,\delta+1}}
\right).
\end{equation*} 
\end{rem}

\section{Energy estimates}
\label{sec:energy-estimates}

In this section we will derive the energy estimates for
a linear symmetric hyperbolic system, a system which we have obtained by
linearising  (\ref{eq:3.4}).
So we consider 
\begin{equation}
\label{eq:5.1}
 \mathcal{A}^0\partial_t U=\left(\mathcal{A}^a+\mathcal{C}^a\right)\partial_a
U+\mathcal{B}\left(
\begin{array}{l}
v\\
\partial_t v\\ \partial_x v
\end{array}\right)
+ \mathcal{F}+\mathcal{D},
\end{equation} 
where $U=(v,\partial_t v,\partial_x v,W)$,  the matrices 
$\mathcal{A}^0$, $\mathcal{A}^a$, $\mathcal{B}$ and $\mathcal{C}^a$  have the
same structural form as the corresponding matrices in (\ref{eq:3.4}),
$\mathcal{C}^a$ is a constant matrix, and the vectors
$\mathcal{F}$ and $\mathcal{D}$ have  the form $(0,f,0,0)$  and $(0,d_2,
d_3,d_4)$ respectively.
     
\begin{con}
  \label{ass:1}
  All the matrices have the same block structure as (\ref{eq:3.7}),
(\ref{eq:3.9})
  and (\ref{eq:3.6}) and satisfy:
  \begin{subequations}
    \begin{eqnarray}
      &&\ 
      \left(\mathcal{A}^0(t,\cdot)-{\bf e}_{55}\right),
\mathcal{A}^a(t,\cdot)\in
      H_{s+1,\delta};\\ &&\ \
      \label{eq:5.3}
\exists \  c_0 \geq1 \ \mbox{such that }\
      {c_0}^{-1}V^TV\leq V^T\mathcal{A}^0V\leq c_0 V^TV,\quad \forall V\in
\setR^{55};\ 
      \\ &&\ \
      \label{eq:5.9}
      \partial_t \mathcal{A}^0(t,\cdot)\in L^\infty;\\
      &&\ \
      {\bf b}_{2j}(t,\cdot)\protect {, {\bf b}_{4j}(t,\cdot)}\in
H_{s,\delta+1}, \ \ \ j=2,3,4,5;
      \\
      &&      \mathcal{F}(t,\cdot),  \mathcal{D}(t,\cdot)\in H_{s,\delta+1}.
    \end{eqnarray}
  \end{subequations}
\end{con}

\subsection{$X_{s,\delta}$--energy estimates}
\label{sec:x_s-delta-energy}

We turn now to the definition of an inner--product of the space $X_{s,\delta}$
which takes into account the structure of the matrix $\mathcal{A}^0$ of the
system (\ref{eq:5.1}).

Let $F(u)$ denote the Fourier transform of a distribution  $u$ and set
\begin{equation*}
 \Lambda^s(u)=
(1-\Delta)^{\frac{s}{2}}=F^{-1}\left(\left(1+|\xi|^2\right)^{\frac{s}{2}}
F\right)(u).
\end{equation*}
The standard inner--product of the Bessel-potential spaces $H^s$ is
\begin{equation*}
 \langle u_1,u_2\rangle_s =\langle
\Lambda^s(u_1),\Lambda^s(u_2)\rangle_{L^2}.
\end{equation*}
Taking into account the term--wise definition of the norm (\ref{eq:weighted:4}),
we define the  inner--product of $H_{s,\delta}$ as follows:
 \begin{equation}
\label{eq:5.2}
       \langle u_1,u_2\rangle_{s,\delta} :=\sum_{j=0}^\infty
       2^{\left(\delta+\frac{3}{2}\right)2j}\left\langle
         \Lambda^{s}\left(\psi_j^2
           u_1\right)_{2^j}, \Lambda^{s}\left(\psi_j^2
           u_2\right)_{2^j}\right\rangle_{L^2},
\end{equation}
where $(u)_{2^j}$ denotes scaling by $2^j$.  By Proposition \ref{equivalence:1}, $ \langle
u,u\rangle_{s,\delta}=\left\|u\right\|_{H_{s,\delta,2}}^2\simeq
\left\|u\right\|_{H_{s,\delta}}^2$.
To each component of the space
\begin{equation*}
 X_{s,\delta}:=H_{s,\delta}\times  H_{s,\delta+1}\times
H_{s,\delta+1} \times
H_{s+1,\delta+1}
\end{equation*}
we assign its own inner--product. Since $\mathcal{A}^0=({\bf a}_{ij}^0)$, where
${\bf a}_{ij}^0$ is the zero matrix for $i\not=j$, ${\bf a}_{ii}^0$ is
 the identity for $i=1,2$, we assign to the first two components the
inner--product (\ref{eq:5.2}), while for the other terms we insert
$\mathcal{A}^0$ termwise.

 \begin{defn}[Inner--product in $X_{s,\delta}$]
Let $U_i=(v_i,\partial_t
v_i,\partial_x v_i, W_i)\in X_{s,\delta}$, $i=1,2$ and assume that the matrix
$\mathcal{A}^0$ satisfies Assumption \ref{ass:1}, then we denote 
the  {\it inner--product of $X_{s,\delta}$}:
\begin{equation}
  \label{eq:5.6}
  \begin{split}
    \left\langle U_1,U_2\right\rangle_{X_{s,\delta,\mathcal{A}^0}}& :=
    \langle   v_1,v_2\rangle_{s,\delta}+  \langle \partial_t v_1,\partial_t
    v_2\rangle_{s,\delta+1} \\ &+ \left\langle\partial_x v_1, \partial_x v_2
    \right\rangle_{s,\delta+1,{\bf a}_{33}^0} +
    \left\langle W_1, W_2\right\rangle_{s+1,\delta+2,{\bf a}_{44}^0},
  \end{split}
\end{equation} 
where the terms are defined in the following way:

\begin{itemize}

  \item {\it An inner--product of $H_{s,\delta}$} of the form:
  \begin{math}
    \langle v_1,v_2\rangle_{s,\delta}
  \end{math}, where the inner--product $\langle
  \cdot,\cdot\rangle_{s,\delta}$ is defined by (\ref{eq:5.2});
  
  \item {\it An inner--product of $ H_{s,\delta+1} $} of the form:
  \begin{math}
    \langle \partial_t v_1,\partial_t v_2\rangle_{s,\delta+1}
  \end{math}, where $\langle \cdot,\cdot\rangle_{s,\delta+1} $ is
  defined by (\ref{eq:5.2}) with $\delta+1$;

  \item {\it An inner-product on $ H_{s,\delta+1} $} of the form:
  \begin{equation}
    \label{eq:5.8}
    \begin{split}
      \left\langle
        \partial_x v_1, \partial_x v_2\right\rangle_{s,\delta+1,{\bf
          a}_{33}^0} := \sum_{j=0}^\infty
      2^{\left(\delta+1+\frac{3}{2}\right)2j}\left\langle
        \Lambda^s\left(\psi_j^2 \partial_x v_1\right)_{2^j}, \left(
          {\bf a}_{33}^0 \right)_{2^j} \Lambda^s\left(\psi_j^2
          \partial_x v_2 \right)_{2^j}\right\rangle_{L^2};
    \end{split}
  \end{equation}

  \item {\it An inner--product of $H_{s+1,\delta+1}$} of
  the form:
  \begin{equation}
    \label{eq:5.11}
    \begin{split}
      \left\langle W_1, W_2\right\rangle_{s+1,\delta+2,{\bf a}_{44}^0}
      := \sum_{j=0}^\infty
      2^{\left(\delta+1+\frac{3}{2}\right)2j}\left\langle
        \Lambda^{s+1}\left(\psi_j^2 W_1\right)_{2^j}, \left( {\bf
            a}_{44}^0\right)_{2^j} \Lambda^{s+1}\left(\psi_j^2
          W_2\right)_{2^j}\right\rangle_{L^2};
    \end{split}
  \end{equation}
\end{itemize}
\end{defn}
We denote by $\left\| U\right\|_{X_{s,\delta,\mathcal{A}^0}}$ the norm
associated with the inner--product (\ref{eq:5.6}). Since the matrix
$\mathcal{A}^0$ satisfies  (\ref{eq:5.3}), the
following equivalence
\begin{equation}
\label{eq:5.4}
 \left\|U\right\|_{X_{s,\delta}}\lesssim
\left\|U\right\|_{X_{s,\delta,\mathcal{A}^0}}\lesssim
\left\|U\right\|_{X_{s,\delta}} 
\end{equation}
holds. In order to simplify the notation we set $U(t)=U(t,x^1,x^2,x^3)$. 
\vskip 5mm
\begin{lem}
\label{lem:1}
 Let $s>\frac{3}{2}$,
$\delta\geq-\frac{3}{2}$ and assume the coefficients of
(\ref{eq:5.1})  satisfy Assumptions \ref{ass:1}. If $U(t)\in
C_0^\infty(\setR^3)$ is a solution of (\ref{eq:5.1}), then
\begin{equation}
\label{eq:5.5}
 \dfrac{d}{dt} \left\langle
U(t),U(t)\right\rangle_{X_{s,\delta,\mathcal{A}^0}}\leq
Cc_0 \left( \left\langle
U(t),U(t)\right\rangle_{X_{s,\delta,\mathcal{A}^0}}+1\right),
\end{equation} 
where the constant $C$ depends on the corresponding  norms of the
coefficients, $s$ and $\delta$.  
\end{lem}
The corresponding energy estimates for the vacuum Einstein equations
were obtained in \cite{Ka1}. The same techniques can be applied here with some
obvious modifications.  We therefore give only a short sketch of the proof.
\vskip 5mm
\noindent
\textit{Sketch of the proof.}
From the inner--product (\ref{eq:5.6}) we see that
\begin{equation*}
\begin{split}
 \dfrac{1}{2}\dfrac{d}{dt} \left\langle
U,U\right\rangle_{X_{s,\delta,\mathcal{A}^0}} &= \langle
v,\partial_t v\rangle_{s,\delta}+  \langle
\partial_t v,\partial_t^2 v\rangle_{s,\delta+1}+\left\langle  \partial_x
v, \partial_x\partial_t v\right\rangle_{s,\delta+1,{\bf a}_{33}^0} \\ & +
\left\langle W,\partial_t W\right\rangle_{s+1,\delta+2,{\bf a}_{44}^0} \\  &+
\sum_{j=0}^\infty
2^{\left(\delta+1+\frac{3}{2}\right)2j}\left\langle
\Lambda^s\left(\psi_j^2 \partial_x v\right)_{2^j}, \partial_t\left( {\bf
a}_{33}^0\right)_{2^j}\Lambda^s\left(\psi_j^2 \partial_x
v\right)_{2^j}\right\rangle_{L^2} \\  &+
\sum_{j=0}^\infty
2^{\left(\delta+2+\frac{3}{2}\right)2j}\left\langle
\Lambda^{s+1}\left(\psi_j^2 W\right)_{2^j}, \partial_t\left( {\bf
a}_{44}^0\right)_{2^j}\Lambda^{s+1}\left(\psi_j^2
W\right)_{2^j}\right\rangle_{L^2}.
\end{split}
\end{equation*}
By the Cauchy Schwarz inequality, we obtain
\begin{equation*}
 |\langle v,\partial_t v\rangle_{s,\delta}|\leq \left\|
v\right\|_{H_{s,\delta,2}}\left\| \partial_t
v\right\|_{H_{s,\delta,2}}\leq\dfrac{1}{2}\left(\left\|
v\right\|_{H_{s,\delta,2}}^2+\left\|\partial_t
v\right\|_{H_{s,\delta+1,2}}^2\right),
\end{equation*}
and by Assumption (\ref{eq:5.9}), the first infinite sum is less than
\begin{equation*}
\begin{split}
 &C\left\|\partial_t {\bf a}_{33}^0\right\|_{L^\infty}
\sum_{j=0}^\infty
2^{\left(\delta+1+\frac{3}{2}\right)2j}\left\langle
\Lambda^s\left(\psi_j^2 \partial_x v\right)_{2^j},\Lambda^s\left(\psi_j^2
\partial_x
v\right)_{2^j}\right\rangle_{L^2}\\ =& C\left\|\partial_t {\bf
a}_{33}^0\right\|_{L^\infty}\left\|\partial_x v\right\|_{H_{s,\delta+1,2}}^2.
\end{split}
\end{equation*}
A similar estimate  holds for the second infinite sum. 
The most difficult part is the estimate 
\begin{equation}
\label{eq:5.12}
 \Big|  \langle
\partial_t v,\partial_t^2 v\rangle_{s,\delta+1}+\left\langle  \partial_x
v, \partial_x\partial_t v\right\rangle_{s,\delta+1,{\bf a}_{33}^0} \Big|\lesssim
\left(\left\| U\right\|_{X_{s,\delta}}^2 +1\right),
\end{equation}
and here it is essential to use the assumption that $\mathcal{A}^\alpha\in
H_{s+1,\delta}$ and $s>\frac{3}{2}$. 
We present here the main ideas of this estimate and for a detailed proof 
we refer to \cite[\S4]{Ka1}.

Let
\begin{equation}
\label{eq:5.10}
 E_{\partial_t}(j)=
\left\langle
           \Lambda^s\left(\psi_j^2 \partial_t v\right)_{2^j},
           \Lambda^s\left(\psi_j^2 (\partial_t^2 v)
           \right)_{2^j}\right\rangle_{L^2}
\end{equation} 
and 
\begin{equation}
\label{eq:5.14}
 E_{\partial_x}(j)=
\left\langle
           \Lambda^s\left(\psi_j^2 \partial_x v\right)_{2^j},
           \left({\bf a}_{33}^0\right)_{2^j}
           \Lambda^s\left(\psi_j^2 \partial_t(\partial_x v)
           \right)_{2^j}\right\rangle_{L^2}.
\end{equation} 
In order to use equation  (\ref{eq:5.1}) we need to commute $\left({\bf
a}_{33}^0\right)_{2^j}$ both with the operator $\Lambda^s$ and $\psi_j^2$. 
An essential ingredient is the Kato--Ponce commutator estimate 
(\ref{eq:appen:9}) (see Appendix). Usually
this estimate is used  in similar situations with the operator
$\Lambda^s$, 
we, however, apply  it to the pseudodifferential
operator $P=\Lambda^s\partial_x$.
This  enables  us to
use  estimate (\ref{eq:appen:9}) with the index $s+1$ and to exploit the
assumption that $\mathcal{A}^\alpha\in H^{s+1}_{\rm loc}$.
 
Let $\Psi_k=\left(\sum_{j=0}^\infty\psi_j\right)^{-1}\psi_k$, then ${\rm
supp}(\Psi_k\psi_j)\neq\emptyset$  for $k=j-2,\ldots,j+1$, and hence we have
that
\begin{equation}
\label{eq:5.112}
 E_{\partial_x}(j)=\sum_{k=j-2}^{j+1}\left\langle
           \Lambda^s\left(\psi_j^2 \partial_x v\right)_{2^j},
           \left({\bf a}_{33}^0\right)_{2^j}
           \Lambda^s\left(\Psi_k\psi_j^2 \partial_t(\partial_x v)
           \right)_{2^j}\right\rangle_{L^2}=:\sum_{k=j-2}^{j+1}
E_{\partial_x}(j,k),
\end{equation} 
and similarly
\begin{equation}
\label{eq:5.113}
 E_{\partial_t}(j)=\sum_{k=j-2}^{j+1}\left\langle
           \Lambda^s\left(\psi_j^2 \partial_t v\right)_{2^j},
           \Lambda^s\left(\Psi_k\psi_j^2 \partial_t(\partial_t v)
           \right)_{2^j}\right\rangle_{L^2}=:\sum_{k=j-2}^{j+1}
E_{\partial_t}(j,k).
\end{equation} 
Writing 
\begin{equation}
 \left(\Psi_k\psi_j^2 \partial_t(\partial_x
v)\right)_{2^j}=2^{-j}\partial_x\left(\Psi_k\psi_j^2\partial_t v\right)_{2^j}-
\partial_x\left(\Psi_k\psi_j^2\right)_{2^{j}}
\left(\partial_t v\right)_{2j}, 
\end{equation} 
and
\begin{equation}
\label{eq:5.13}
\begin{split}
 \Lambda^s\left(\partial_x\left(\Psi_k\psi_j^2\partial_t v\right)_{2^j}\right)
& =
\left(\Lambda^s\partial_x\right)\left(\Psi_k\psi_j^2\partial_t v\right)_{2^j}
-\left(\Psi_k\right)_{2^j}
\left(\Lambda^s\partial_x\right)\left(\left(\psi_j^2\partial_t
v\right)_{2^j}\right)\\ & +
\left(\Psi_k\right)_{2^j}
\left(\Lambda^s\partial_x\right)\left(\left(\psi_j^2\partial_t
v\right)_{2^j}\right),
\end{split}
\end{equation} 
we obtain that
\begin{equation}
\begin{split}
 E_{\partial_x}(j,k) & =2^{-j}\left\langle \Lambda^s\left(\psi_j^2\partial_x v
\right), \left(\Psi_k{\bf
a}_{33}^0\right)_{2^j}
\left(\Lambda^s\partial_x\right)\left(\psi_j^2\partial_t
v\right)_{2^j}\right\rangle +R(a,j,k)\\ &=: E_{\partial_x}(a,j,k)+ R(a,j,k). 
\end{split}
\end{equation}
We estimate the remainder term $R(a,j,k)$ by using the 
Cauchy--Schwarz
inequality, the  Sobolev embedding theorem and of course the Kato--Ponce commutator
estimate. This results in
\begin{equation}
\label{eq:5.17}
 |R(a,j,k)|\lesssim \left\|\left({\bf a}_{33}^0\right)_{2^j}\right\|_{L^\infty}
\left\|\left(\psi_j^2\partial_x v\right)_{2^j}\right\|_{H^s}\left(
\left\|\left(\psi_j\partial_t v\right)_{2^j}\right\|_{H^s}+
\left\|\left(\psi_j^2\partial_t v\right)_{2^j}\right\|_{H^s}\right).
\end{equation}

Next, we commute $\left(\Psi_k{\bf a}_{33}^0\right)_{2^j}$ with
$\Lambda^s\partial_x$, that
is, we write, 
\begin{equation}
 \begin{split}
&\left(\Psi_k{\bf a}_{33}^0\right)_{2^j}
\left(\Lambda^s\partial_x\right)\left(\psi_j^2\partial_t v\right)_{2^j} =
\left(\Psi_k{\bf a}_{33}^0\right)_{2^j}
\left(\Lambda^s\partial_x\right)\left(\psi_j^2\partial_t v\right)_{2^j}-
\left(\Lambda^s\partial_x\right)\left(\left(\Psi_k{\bf a}_{33}^0\right)_{2^j}
\left(\psi_j^2\partial_t v\right)_{2^j}\right)\\ &+
\Lambda^s\left(\left(\partial_x\left(\Psi_k{\bf
a}_{33}^0\psi_j^2\right)\right)_{2^j}\left(\partial_t v\right)_{2^j} \right)+2^j
\Lambda^s\left(\left(\Psi_k{\bf
a}_{33}^0\psi_j^2\right)_{2^j}\left(\partial_t\partial_x v\right)_{2^j}
\right),
 \end{split}
\end{equation} 
then
\begin{equation}
\begin{split}
 E_{\partial_x}(a,j,k) &=
\left\langle \Lambda^s\left(\psi_j^2\partial_x v \right),
\Lambda^s\left(\Psi_k{\bf a}_{33}^0\psi_j^2\partial_t
\partial_x v\right)_{2^j}\right\rangle_{L^2} + R(b,j,k)\\ & :=
E_{\partial_x}(b,j,k)+R(b,j,k).
\end{split}
\end{equation} 

Since ${\bf a}_{33}^0\in H^{s+1}_{\rm loc}$,  the Kato--Ponce commutator
estimate (\ref{eq:appen:9}) implies that
\begin{equation}
 \begin{split}
  &\left\|\left(\Psi_k{\bf a}_{33}^0\right)_{2^j}
\left(\Lambda^s\partial_x\right)\left(\psi_j^2\partial_t v\right)_{2^j}-
\left(\Lambda^s\partial_x\right)\left(\left(\Psi_k{\bf a}_{33}^0\right)_{2^j}
\left(\psi_j^2\partial_t v\right)_{2^j}\right)\right\|_{L^2}\\ \lesssim & 
\left\|\nabla\left(\Psi_k{\bf a}_{33}^0\right)_{2^j}\right\|_{L^\infty}
\left\|\left(\psi^2_j\partial_t v\right)_{2^j}\right\|_{H^{s}}+
\left\|\left(\Psi_k{\bf a}_{33}^0\right)_{2^j}\right\|_{H^{s+1}}
\left\|\left(\psi^2_j\partial_t v\right)_{2^j}\right\|_{L^\infty}.
 \end{split}
\end{equation}

Taking into account that $\left\|\left(\Psi_k{\bf a}_{33}^0\right)_{2^j}\right\|_{H^{s+1}}\simeq
\left\|\left(\psi_k{\bf a}_{33}^0\right)_{2^k}\right\|_{H^{s+1}}$, we
obtain in a similar manner, as in as in the estimate of the previous
remainder term, the following
\begin{equation}
\label{eq:5.21}
\begin{split}
 |R(b,j,k)| &\lesssim\left(\left\|\left(\psi_k\left({\bf a}_{33}^0-{\bf
e}_{33}\right)\right)_{2^k}\right\|_{H^{s+1}}+1\right)
\left\|\left(\psi_j^2\partial_x v\right)_{2^j}\right\|_{H^{s}}\\ &\times
\left(\left\|\left(\psi_j\partial_t v\right)_{2^j}\right\|_{H^{s}}+
\left\|\left(\psi^2_j\partial_t v\right)_{2^j}\right\|_{H^{s}}\right).
\end{split}
\end{equation}

Adding $E_{\partial_t}(j,k)$ to $E_{\partial_x}(b,j,k) $ enables us now 
 to use equation (\ref{eq:5.1}), that is,
\begin{equation*}
 \begin{split}
  E_{\partial_t}(j,k)+E_{\partial_x}(b,j,k) &=
\left\langle \Lambda^s\left(\psi_j^2\begin{pmatrix}\partial_t v\\ \partial_x v
\end{pmatrix}\right)_{2^j},\Lambda^s\left(\Psi_k\psi_j^2
\begin{pmatrix}{\bf e}_{10} & {\bf 0}_{10\times 30}\\ {\bf 0}_{30\times 10} &
{\bf a}_{33}^0 \end{pmatrix}\partial_t\begin{pmatrix}\partial_t v \\ \partial_x
v\end{pmatrix}\right)_{2^j}\right\rangle_{L^2}\\ &=
\sum_{a=1}^3
\left\langle \Lambda^s\left(\psi_j^2\begin{pmatrix}\partial_t v\\ \partial_x v
\end{pmatrix}\right)_{2^j},\Lambda^s\left(\Psi_k\psi_j^2
\left({\bf a}_{22}^a+{\bf c}_{22}^a\right)\partial_a\begin{pmatrix}\partial_t v
\\ \partial_x
v\end{pmatrix}\right)_{2^j}\right\rangle_{L^2}\\  &+
\left\langle \Lambda^s\left(\psi_j^2\begin{pmatrix}\partial_t v\\ \partial_x v
\end{pmatrix}\right)_{2^j},\Lambda^s\left(\Psi_k\psi_j^2{\bf B}
\right)_{2^j}\right\rangle_{L^2},
 \end{split}
\end{equation*} 
where the matrices ${\bf a}_{22}^a$ and ${\bf c}_{22}^a$ are defined
in subsection \ref{sec:evolution-equations} and ${\bf c}_{22}^a$ are
constant matrices.
In the last expression ${\bf B}$ contains the zero and first order
derivatives of $v$.
It is straightforward to estimate this term since it does not contain
second order derivatives.

In order to estimate the second order terms, we write
\begin{equation}
\label{eq:5.23}
\begin{split}
 &\left(\Psi_k\psi_j^2
\left({\bf a}_{22}^a+{\bf c}_{22}^a\right)\partial_a\begin{pmatrix}\partial_t v
\\ \partial_x v\end{pmatrix}\right)_{2^j}=2^{-j}\partial_a
\left(\Psi_k
\left({\bf a}_{22}^a+{\bf c}_{22}^a\right)\psi_j^2\begin{pmatrix}\partial_t v
\\ \partial_x v\end{pmatrix}\right)_{2^j} \\ -
& \left(\partial_a \left(\Psi_k\psi_j^2
\left({\bf a}_{22}^a+{\bf c}_{22}^a\right)\right)\begin{pmatrix}\partial_t v
\\ \partial_x
v\end{pmatrix}\right)_{2^j}
\end{split}
\end{equation} 
{and then we commute $\Lambda^s\partial_a$ with $\Psi_k
\left({\bf a}_{22}^a+{\bf c}_{22}^a\right)$ in the first term of the right hand
side of (\ref{eq:5.23}).  This}  results in 
\begin{equation}
\label{eq:5.24}
 \left\langle \Lambda^s\left(\psi_j^2\begin{pmatrix}\partial_t v\\ \partial_x v
\end{pmatrix}\right)_{2^j},\partial_a\left(\Psi_k
\left({\bf a}_{22}^a+{\bf c}_{22}^a\right)\right)_{2^j}
\Lambda^s\left(\psi_j^2\begin{pmatrix}\partial_t
v \\ \partial_x v\end{pmatrix}\right)_{2^j}\right\rangle_{L^2},
\end{equation} 
plus a remainder term
 which can be estimate by the Kato--Ponce commutator estimate 
(\ref{eq:appen:9}) in a similar manner to the previous steps.
{We use the common method of integration by parts in order to estimate
the inner--product (\ref{eq:5.24})}, and as to the second term of
(\ref{eq:5.23}),
we have that 
\begin{equation*}
 \left(\partial_a \left(\Psi_k\psi_j^2
\left({\bf a}_{22}^a+{\bf c}_{22}^a\right)\right)\begin{pmatrix}\partial_t v
\\ \partial_x
v\end{pmatrix}\right)=\left(\partial_a\left(\Psi_k\left({\bf
a}_{22}^a+{\bf c}_{22}^a\right)\right)\psi_j+2\Psi_k\left({\bf
a}_{22}^a+{\bf
c}_{22}^a\right)\partial_a\psi_i\right)\psi_j\begin{pmatrix}\partial_t
v\\ \partial_x v\end{pmatrix}.
\end{equation*}
So by the Cauchy--Schwraz inequality we obtain that
\begin{equation}
\label{eq:5.25}
\begin{split}
&\Big| \left\langle \Lambda^s\left(\psi_j^2\begin{pmatrix}\partial_t v\\
\partial_x v
\end{pmatrix}\right)_{2^j},\Lambda^s \left(\partial_a\left(\Psi_k\left({
\bf a}_{22}^a\right)\right)\psi_j^2\begin{pmatrix}\partial_t
v \\ \partial_x v\end{pmatrix}\right)_{2^j}\right\rangle_{L^2}\Big| \\ \lesssim
&
\left(\left\|\left(\psi_j^2 \partial_t v\right)_{2^j}\right\|_{H^s}+
\left\|\left(\psi_j^2 \partial_x v\right)_{2^j}\right\|_{H^s}\right)\left\|
\left(\partial_a\left(\Psi_k\left({
\bf a}_{22}^a\right)\right)\psi_j^2\begin{pmatrix}\partial_t
v \\ \partial_x v\end{pmatrix}\right)_{2^j}\right\|_{H^s},
\end{split}
\end{equation}
and since $\mathcal{A}^a\in H_{s+1,\delta}, s>\frac{3}{2}$, we have by
the multiplicity property of the $H^s$ spaces that
\begin{equation}
\label{eq:5.26}
\begin{split}
& \left\|
\left(\partial_a\left(\Psi_k\left({
\bf a}_{22}^a\right)\right)\psi_j^2\begin{pmatrix}\partial_t
v \\ \partial_x v\end{pmatrix}\right)_{2^j}\right\|_{H^s}\\ \lesssim &
\left\|\left(\psi_k{
\bf a}_{22}^a\right)_{2^j}\right\|_{H^{s+1}}
\left(\left\|\left(\psi_j^2 \partial_t v\right)_{2^j}\right\|_{H^s}+
\left\|\left(\psi_j^2 \partial_x v\right)_{2^j}\right\|_{H^s}\right).
\end{split}
\end{equation}

The other term can be estimated by similar arguments.
From inequalities (\ref{eq:5.17}), (\ref{eq:5.21}), (\ref{eq:5.25}) and
(\ref{eq:5.26}), we conclude that
\begin{equation}
\begin{split}
 &\sum_{j=0}^\infty\sum_{k=j-2}^{j+1}
2^{(\delta+1+\frac{3}{2})2j}\Big|E_{\partial_t}(j,k)+
E_{\partial_x}(j,k)\Big|\\ \lesssim & \sum_{j=0}^\infty\sum_{k=j-2}^{j+1}
2^{(\delta+1+\frac{3}{2})2j}
\left\|\left(\psi_k {\bf a}_{22}^a\right)_{2^k}\right\|_{H^{s+1}}
\left\|\left(\psi_j^2 \partial_t v\right)_{2^j}\right\|_{H^{s}}
\left\|\left(\psi_j \partial_x v\right)_{2^j}\right\|_{H^{s}}+\cdots,
\end{split}
\end{equation}  
where $\cdots$ represent terms that are sums of similar structure.
  For example $\psi_j$ is replaced by $\psi_j^2$, ${\bf a}_{22}^a$ is
  replaced by $({\bf a}_{33}^0-{\bf e}_{33})$, or the $H^{s+1}$--norm
  is replaced by the $L^\infty$--norm.

Applying the H\"older inequality and using the equivalence of norms
(Propositions \ref{equivalence:1}), we have that 
\begin{equation}
\begin{split}
 &\sum_{j=0}^\infty\sum_{k=j-2}^{j+1}
2^{(\delta+1+\frac{3}{2})2j}
\left\|\left(\psi_k {\bf a}_{22}^a\right)_{2^k}\right\|_{H^{s+1}}
\left\|\left(\psi_j^2 \partial_t v\right)_{2^j}\right\|_{H^{s}}
\left\|\left(\psi_j \partial_x v\right)_{2^j}\right\|_{H^{s}}\\ 
\lesssim  & \left\|{\bf a}_{22}^a\right\|_{H_{s+1,\delta}}\left\|\partial_t
v\right\|_{H_{s+1,\delta}} 
\left\|\partial_x v\right\|_{H_{s+1,\delta}}\leq
\left\|{\bf a}_{22}^a\right\|_{H_{s+1,\delta}}\left(\left\|\partial_t
v\right\|_{H_{s+1,\delta}}^2 +
\left\|\partial_x v\right\|_{H_{s+1,\delta}}^2\right).
\end{split}
\end{equation}

So recalling the properties of  the inner--products 
(\ref{eq:5.2}), (\ref{eq:5.8}), (\ref{eq:5.6}), (\ref{eq:5.10}),
(\ref{eq:5.14}), (\ref{eq:5.112}) and (\ref{eq:5.113}), we obtain
(\ref{eq:5.12}).
 The estimate of $\left\langle W,\partial_t W\right\rangle_{s+1,\delta+2,{\bf
a}_{44}^0}$ relies on similar ideas to those of (\ref{eq:5.12}), but it is
simpler, since  $W\in H^{s+1}_{\rm loc}$.
Having collected the estimates  of all the terms, we have 
\begin{equation*}
  \dfrac{1}{2}\dfrac{d}{dt} \left\langle
U,U\right\rangle_{X_{s,\delta,\mathcal{A}^0}}\lesssim \left(\left\| U
\right\|_{X_{s,\delta}}^2 + 1\right)
\end{equation*}
and by the equivalence (\ref{eq:5.4}) we obtain (\ref{eq:5.5}).
\hfill{$\square$}

\begin{rem} 
Note that if the coefficients of $\partial_t v$ in the matrix $\mathcal{A}^0$
were dependent on $t$, then  we would have to reiterate the commutation
(\ref{eq:5.13}), but with the operator $\Lambda^s\partial_t$ instead of   
$\Lambda^s\partial_x$. However,  $\Lambda^s\partial_t$ is pseudodifferential 
operator of oder $s$, and hence we would not get the desired regularity. This
is the reason for dividing equations (\ref{eq:2.3}) by $-g^{00}$. 
\end{rem}

\subsection{$L_\delta^2$--energy estimates}
\label{sec:l_delt-energy-estim}

The next section deals with the existence of solutions to the
nonlinear symmetric hyperbolic systems by means of an iteration
scheme.
The $X_{s,\delta}$--energy estimates are used to obtain boundedness of
the sequence, while $L_\delta^2$--energy estimates are needed in order
to establish the contraction.

Let 
\begin{equation*}
 \langle
u_1,u_2\rangle_{\delta}=\int_{\setR^3}\left(1+|x|\right)^{2\delta}
u_1^T(x)u_2(x)dx,
\end{equation*} 
denote a weighted $L^2$ inner--product,  where $u_1^Tu_2$ denote the scalar
product between two vectors.
The $L_\delta^2(\setR^3)$--space is the closure of all continuous functions
under the norm $\|u\|_{L_\delta^2}^2=\langle u,u\rangle_\delta$, and this norm
 is equivalent to the norm $\|u\|_{H_{0,\delta}}$ (see
\cite{triebel76:_spaces_kudrj2}). Similar to (\ref{eq:5.7}), we set
\begin{equation*}
 Y_\delta = L_\delta^2\times L_{\delta+1}^2\times L_{\delta+1}^2\times
L_{\delta+1}^2
\end{equation*}
and $\|U\|_{Y_\delta}^2= \left\| v\right\|_{L_\delta^2}^2+ \left\|
\partial_t v\right\|_{L_{\delta+1}^2}^2+\left\|
\partial_x v\right\|_{L_{\delta+1}^2}^2+\left\|
W\right\|_{L_{\delta+1}^2}^2$.
We also define the  inner--product of $Y_\delta$ in
accordance with to the system (\ref{eq:5.1}):
\begin{equation*}
\begin{split}
& \left\langle U_1,U_2\right\rangle_{Y_\delta,\mathcal{A}^0}=
\left\langle v_1,v_2\right\rangle_\delta+\left\langle
\partial_t v_1,\partial_t v_2\right\rangle_{\delta+1}
+\left\langle\partial_x v_1,{\bf a}_{33}^0\partial_x v_2\right\rangle_{\delta+1}
+\left\langle W_1,{\bf a}_{44}^0 W_2\right\rangle_{\delta+1},
\end{split}
\end{equation*} 
and the associated norm $\left\|
U\right\|_{Y_\delta,\mathcal{A}^0}^2=\langle
U,U\rangle_{Y_\delta,\mathcal{A}^0}$. By assumption (\ref{eq:5.3}), $\left\|
U\right\|_{Y_\delta,\mathcal{A}^0}\simeq \left\|U\right\|_{Y_\delta}$.

\begin{lem}
  \label{lem:2}
  Assume the coefficients of (\ref{eq:5.1}) satisfy Assumptions
  \ref{ass:1}. If $U(t)\in X_{1,\delta}$ is a solution of (\ref{eq:5.1}),
  then
  \begin{equation}
\label{eq:5.29}
    \dfrac{d}{dt} \left\langle
U(t),U(t)\right\rangle_{Y_\delta,\mathcal{A}^0}\leq
    Cc_0 \left(\left\langle
        U(t),U(t)\right\rangle_{Y_\delta,\mathcal{A}^0}+\left\|\mathcal{F}
      \right\|_{L_{\delta+1}^2}^2+\left\|\mathcal{D}
      \right\|_{L_{\delta+1}^2}^2\right) ,
  \end{equation} 
  where the constant $C$ depends upon the $L^\infty$--norms of $
  \mathcal{A}^\alpha$, $\partial_\alpha \mathcal{A}^\alpha$ and $\mathcal{B}$.
\end{lem}

\begin{proof}
 We find that
\begin{equation*}
 \begin{split}
  \frac{1}{2}\frac{d}{dt} &\left\langle
U(t),U(t)\right\rangle_{Y_\delta,\mathcal{A}^0} =\left\langle
v,\partial_t v\right\rangle_\delta+\left\langle
\partial_t v,\partial_t^2 v\right\rangle_{\delta+1}
+\left\langle\partial_x v,{\bf a}_{33}^0 \partial_t \partial_x
v\right\rangle_{\delta+1} \\ & +
\left\langle W,{\bf a}_{44}^0 \partial_t W\right\rangle_{\delta+1}
+\left\langle\partial_x v,\partial_t({\bf a}_{33}^0)  \partial_x
v\right\rangle_{\delta+1}
+\left\langle W,\partial_t({\bf a}_{44}^0)  W\right\rangle_{\delta+1}.
 \end{split}
\end{equation*}  
By the Cauchy--Schwarz inequality
$|\left\langle v,\partial_t
v\right\rangle_\delta|\leq(\left\|v\right\|_{L_{\delta}^2}^2+
\left\|\partial_t v\right\|_{L_{\delta}^2}^2)$, 
$|\left\langle\partial_x v,\partial_t({\bf a}_{33}^0)  \partial_x
v\right\rangle_{\delta+1}|\lesssim \|\partial_t({\bf a}_{33}^0)\|_{L^\infty}
\|\partial_x v\|_{L_{\delta+1}^2}^2$  and 
$|\left\langle W,\partial_t({\bf a}_{44}^0) 
W\right\rangle_{\delta+1}|\lesssim \|\partial_t({\bf a}_{44}^0)\|_{L^\infty}
\|W\|_{L_{\delta+1}^2}^2$. 
Since the system (\ref{eq:5.1}) is semi-decoupled, we may estimate the 
expressions with $\partial_t v$ and $ \partial_x v$ first, and
later the term $W$. Using
equation (\ref{eq:5.1})
and recalling the structure  matrices (\ref{eq:3.7}), (\ref{eq:3.9}) and
(\ref{eq:3.6}),
we have
\begin{equation*}
 \begin{split}
 &\left\langle\partial_t v,\partial_t^2 v\right\rangle_{\delta+1}
+\left\langle\partial_x v,{\bf a}_{33}^0 \partial_t \partial_x
v\right\rangle_{\delta+1} =\sum_{a=1}^3\left\langle\begin{pmatrix}\partial_t
v\\ \partial_x v\end{pmatrix}, \left({\bf a}_{22}^a+{\bf c}_{22}^a\right)
\partial_a\begin{pmatrix}\partial_t v\\ \partial_x
v\end{pmatrix}\right\rangle_{\delta+1}\\ &+\left\langle\partial_t v, {\bf
b}_{22}\partial_t v+\left( {\bf b}_{23}+{\bf b}_{24}+{\bf
b}_{25}\right)\partial_x v\right\rangle_{\delta+1}+\left\langle\partial_t v,
f+d_2\right\rangle_{\delta+1}+\left\langle\partial_x v,
d_3\right\rangle_{\delta+1}.
 \end{split}
\end{equation*}

Exploring the symmetry of the matrices and using integration by parts, we have
that 
\begin{equation*}
\begin{split}
 &2\left\langle\begin{pmatrix}\partial_t
v\\ \partial_x v\end{pmatrix}, \left({\bf a}_{22}^a+{\bf c}_{22}^a\right)
\partial_a\begin{pmatrix}\partial_t v\\ \partial_x
v\end{pmatrix}\right\rangle_{\delta+1}=-\left\langle\begin{pmatrix}\partial_t
v\\ \partial_x v\end{pmatrix},\partial_a \left({\bf a}_{22}^a\right)
\begin{pmatrix}\partial_t v\\ \partial_x
v\end{pmatrix}\right\rangle_{\delta+1}\\&-2(\delta+1) 
\left\langle\begin{pmatrix}\partial_t
v\\ \partial_x v\end{pmatrix},\frac{x^a}{|x|} \left({\bf a}_{22}^a+{\bf
c}_{22}^a\right)
\begin{pmatrix}\partial_t v\\ \partial_x
v\end{pmatrix}\right\rangle_{\delta+1}.
\end{split}
\end{equation*} 
Thus
\begin{equation*}
\begin{split}
 &\Big|\left\langle\partial_t v,\partial_t^2 v\right\rangle_{\delta+1}
+\left\langle\partial_x v,{\bf a}_{33}^0 \partial_t \partial_x
v\right\rangle_{\delta+1}\Big|\\ \lesssim &
\left(\sum_{a=1}^3\left(\|\partial_a({\bf a}_{22}^a)\|_{L^\infty}+ \|{\bf
a}_{22}^a\|_{L^\infty}\right)
+\|\mathcal{B}\|_{L^\infty}\right)\left(\|\partial_t
v\|_{L_{\delta+1}^2}^2+\|\partial_x v\|_{L_{\delta+1}^2}^2\right) 
\\ + &\|\partial_t v\|_{L_{\delta+1}^2}^2 +\|\partial_x
v\|_{L_{\delta+1}^2}^2
+\|f\|_{L_{\delta+1}^2}^2+\|d_2\|_{L_{\delta+1}^2}^2+\|d_3\|_{L_{\delta+1}^2}^2.
\end{split}
\end{equation*}

And for $W$ we have,
\begin{equation*}
 \left\langle W,{\bf a}_{44}^0 \partial_t W\right\rangle_{\delta+1}=
\sum_{a=1}^3
\left\langle W,{\bf a}_{33}^a \partial_a W\right\rangle_{\delta+1}
+\left\langle W,d_4\right\rangle_{\delta+1},
\end{equation*} 
so similar arguments give that 
\begin{equation*}
 \Big|\left\langle W,{\bf a}_{44}^0 \partial_t W\right\rangle_{\delta+1}\Big|
\lesssim \left(\sum_{a=1}^3\left(\|\partial_a({\bf a}_{33}^a)\|_{L^\infty}+
\|{\bf
a}_{33}^a\|_{L^\infty}\right)+1\right)\|W\|_{L_{\delta+1}^2}^2+\|d_4\|_{L_ {
\delta+1 }^2} ^2.
\end{equation*} 
Finally,  using the equivalence of the norms we obtain (\ref{eq:5.29}).
\end{proof}

\section{The iteration process}
\label{sec:iteration}
In this section we adopt Majda's iterative
scheme\cite{majda84:_compr_fluid_flow_system_conser} in order to prove
the well-posdeness of the coupled hyperbolic system (\ref{eq:3.4}) in
the $H_{s,\delta}$-spaces.
A similar approach was carried out in \cite{Ka1} for the vacuum
Einstein equations, but in the presence of a perfect fluid there are
additional difficulties.
The zero order term $B$ is not necessarily a
$C^\infty$--function since it contains the fractional power of the
Makino variable $w^{2/(\gamma-1)}$.
Hence we could not apply the standard iteration scheme in order to
prove the existence theorem for symmetric
hyperbolic systems.
We denote the initial data by $U_0=(\phi,\varphi,\partial_x\phi,W_0)$,
where $W_0=(w_0, (u_0)^\alpha-e_0^\alpha)$ represents the initial data
for the Makino variable $w$ and the four velocity vector
$u^\alpha-e_0^\alpha$.
\begin{thm}
  \label{thm:1}
  Let $\frac{3}{2}<s<\frac{2}{\gamma-1}+\frac{1}{2}$,
  $-\frac{3}{2}\leq\delta$.
  Assume $U_0\in X_{s,\delta}$, $w_0\geq0$ and that there exits a
  positive constant $\mu$ such that
  \begin{equation}
    \label{eq:6.5}
    \dfrac{1}{\mu}V^TV\leq V^T\mathcal{A}^0(\phi,W_0)V\leq \mu V^TV\ \ \text{for
      all}\
    V\in\setR^{55}.
  \end{equation}
  Then there exists a positive constant $T$ and a unique solution
  $U(t)=(v(t),\partial_t v(t),\partial_x v(t), W(t))$ to the system
  (\ref{eq:3.4}) such that $U(0,x)=U_0(x)$,
  \begin{equation}
    \label{eq:6.4}
    U\in C\left([0,T], X_{s,\delta}\right)\ \ \ \text{and}\ \ \ W\in
    C^1\left([0,T], H_{s,\delta+2}\right).
  \end{equation} 
  The size of the time interval depends only on the norms of the initial data.
\end{thm}
The proof of Theorem \ref{thm:1} will be carried out in several steps:
\begin{enumerate}
 \item[{\bf1.}] Setting up  the iterative scheme;
\item[{\bf2.}]  Proving that the fractional power $(w^k)^{2/(\gamma-1)}$ is a
$C_0^\infty$-function;
\item[{\bf3.}] Boundedness of the iteration sequence in the $X_{s,\delta}$-norm;
\item[{\bf4.}] Contraction in a lower norm;
\item[{\bf5.}] Convergence;
\item[{\bf6.}] Uniqueness;
\item[{\bf7.}] Continuity in the $X_{s,\delta}$--norm.
\end{enumerate}

A part of the above proofs are standard, but some of them require a
special attention due to the specific form of the system
(\ref{eq:3.4}) and use of the space $X_{s,\delta}$.
Moreover, the fact that the matrices
$\mathcal{A}^\alpha=\mathcal{A}^\alpha(v,W)$ are not dependent on the
derivative of the semi-metric plays an essential role here.
\vskip 5mm
\noindent
\textbf{Step 1.}
From condition (\ref{eq:6.5}) and the embedding into the continuous,
Proposition \ref{Embedding}, we see that there is a bounded domain
$G\subset \setR^{55}$ containing the initial value $U_0$ and a
constant $c_0\geq1$ such that
\begin{equation}
\label{eq:6.6}
 \dfrac{1}{c_0}V^TV\leq V^T\mathcal{A}^0(v,W)V\leq c_0 V^TV
\end{equation}
for all $ U=(v,\pa_t v,\pa_x v, W)\in G$ and $V\in\setR^{55}$.
By means of the density properties of $H_{s,\delta}$, Proposition
\ref{Density}, there 
exist positive constants $C$ and $R$, and a sequence
\begin{equation}
\label{eq:6.16}
\left\{U_0^k\right\}_{k=0}^\infty=\left\{(\phi^k, \varphi^k,
\partial_x\phi^k,w_0^k, (u_0^\alpha)^k-e_0^\alpha)\right\}_{k=0}^\infty\subset
C_0^\infty(\setR^3),
\end{equation}
such that
\begin{equation}
\label{eq:6.24}
  \left\| U^0_0\right\|_{X_{s+1,\delta}}\leq C  \left\|
  U_0\right\|_{X_{s,\delta}},
\end{equation}
\begin{equation}
  \label{eq:6.9}
  \left\|U- U^0_0\right\|_{X_{s,\delta}}\leq R \Rightarrow U\in
G,
\end{equation}
and
\begin{equation}
\label{eq:6.8}
 \left\| U^k_0 -U_0\right\|_{X_{s,\delta}}\leq \dfrac{R 2^{-k}}{4c_0}.
\end{equation}

The iterative  scheme is defined as follows: let $U^0(t,x)=U^0_0(x)$ and 
$$U^{k+1}(t,x)=(v^{k+1}(t,x),\partial_t v^{k+1}(t,x), \partial_x
v^{k+1}(t,x),W^{k+1}(t,x))$$ be a solution of  the linear Cauchy problem
\begin{equation}
 \label{eq:6.1}
\left\{\begin{array}{l}
\begin{array}{lll}
\mathcal{A}^{0}(v^k,W^k)\partial_t
U^{k+1} &=& \left(\mathcal{A}^a(v^k,W^k)+\mathcal{C}^a\right)\partial_a
U^{k+1} \\ & +&\mathcal{B}(U^k)
\begin{pmatrix}
v^{k+1}\\ \partial_t v^{k+1}\\ \partial_x
v^{k+1}
\end{pmatrix}+\mathcal{F}(v^k,W^k)
\end{array}\\
U^{k+1}(0,x) = U_0^{k+1}(x)
\end{array}\right.,
\end{equation}
where $\mathcal{F}(v^k,W^k)=(0,f(v^k,W^k),0,0)$, $f(v^k,W^k)$ is given by
(\ref{eq:3.10}) and $W^k=(w^k, (u^\alpha)^k-e_0^\alpha)$.
\vskip 5mm
\noindent
\textbf{Step 2.}
The iterative method relies on the fact that solutions of linear
symmetric hyperbolic systems with $C_0^\infty$ coefficients and
initial data, are also $C_0^\infty$.
However, even if $w^k\geq0$ and $w^k\in C^\infty_0$, it does not
guarantee that $(w^k)^{2/(\gamma-1)}$ is a $C_0^\infty$-function.
Since the function $f(v^k,W^k)$ contains the term
$(w^k)^{2/(\gamma-1)}$, we must assure that it is a
$C_0^\infty$-function.
\begin{prop}
\label{prop:6.1}
 Let $u\in H_{s,\delta}$ be non-negative and $\beta>0$. Then there is a
sequence $\{u^k\}\subset C_0^\infty$ such that $u^k\to u$ in the 
$H_{s,\delta}$-norm and $(u^k)^\beta\in C_0^\infty$.
\end{prop}
\begin{proof}
  Let $\varepsilon>0$, then by Proposition \ref{Density} there is
  $u^\varepsilon\in C_ 0^\infty$ with $\left\|u-
    u^\varepsilon\right\|_{H_{s,\delta}}<\varepsilon$.
  Take now a positive number $M$ so that $\supp(u^\varepsilon)\subset
  \{|x|\leq M\}$, and let $\chi_M$ be the cut-off function satisfying
  $\chi_M(x)=1$ for $|x|\leq M$ and $\chi_M(x)=0$ for $|x|\geq M+1$.
  For any positive number $\varrho$, we set
  \begin{equation*}
 u^{\varepsilon,\varrho}(x)=\chi_M(x)\left(u^{\varepsilon}(x)+\varrho\right).
\end{equation*}
Then $(u^{\varepsilon,\varrho})^\beta\in C_0^\infty$, since
$(u^{\varepsilon}+\varrho)>0$ and $\chi_M$ is a cut-off function. Moreover,
$ u^{\varepsilon,\varrho}- u^{\varepsilon}=\chi_M\varrho$, hence
$u^{\varepsilon,\varrho}\to u^{\varepsilon}$ in the $H_{s,\delta}$-norm as
$\varrho\to0$.
\end{proof}
Thus we may assume that $\{w_0^k\}_{k=0}^\infty$, the $C_0^\infty$ approximation
of the initial data of the Makino variable $w_0$ in 
(\ref{eq:6.16}), satisfies $(w_0^k)^{2/(\gamma-1)}\in C_0^\infty$. We turn now
showing that for $t\geq 0$,
\begin{equation*}
\epsilon^{k}(t,x)=\left(w^k\right)^{\frac{2}{\gamma-1}}(t,x) 
\end{equation*}
is also a $C_0^\infty$-function. 
\begin{prop}
\label{prop:6.2}
 For each integer $k\geq 0$, $\epsilon^{k}(t,\cdot)\in C_0^\infty(\setR^3)$.
\end{prop}

\begin{proof}
  We conduct the proof by induction.
  Obviously the statement holds when $k=0$.
  Now the $51^{th}$ row of the system (\ref{eq:6.1}) is equivalent to
  (\ref{eq:eineul:8}), so the linearization of it results in
  \begin{equation}
\label{eq:6.2}
\begin{split}
 (u^0)^k\partial_t\epsilon^{k+1}&+(u^a)^k\partial_a\epsilon^{k+1}
+\epsilon^{k}(1+K
(w^k)^2)P_\alpha^\nu\left(g_{\alpha\beta}^k, (u^\beta)^k\right)\partial_v
(u^\alpha)^{k+1} \\ &+\epsilon^{k}(1+K
(w^k)^2)P_\alpha^\nu\left(g^k_{\alpha\beta},
(u^\beta)^k\right)\left(\Gamma_{\nu\mu}^\alpha\right)^k(u^\mu)^k=0,
\end{split}
\end{equation}
where the $P_\alpha^\nu(\cdot,\cdot)$ is the projection
of equation (\ref{eq:eineul:8}) and
$\left(\Gamma_{\nu\mu}^\alpha\right)^k$ are the Christoffel symbols
with respect to the semi-metric $g^k_{\alpha\beta}$.
It follows from \cite[Theorem 2.6]{BK3}, that $u^0(0,x)\geq 1$, hence
$(u^0)^k(0,x)\geq 1$ and therefore we can divide (\ref{eq:6.2}) by
$(u^0)^k$ and we conclude that $\epsilon^{k+1}$ satisfies a first
order linear equation of the form
\begin{equation}
\label{eq:6.3}
\left\{\begin{array}{l}
 \partial_t \epsilon^{k+1}+b_a(t,x)\partial_a \epsilon^{k+1}+c(t,x)=0\\
\epsilon^{k+1}(0,x)=(w^{k+1}_0(x))^{\frac{2}{\gamma-1}}
\end{array}\right..
\end{equation}

Note that $c(t,x)$ contains the term $\epsilon^{k}$, but this term is
$C_0^\infty$ by the induction hypothesis.
Hence all the coefficients of (\ref{eq:6.3}) are
$C_0^\infty$-functions.
We solve (\ref{eq:6.3})  by means of the characteristic method. So let
$\Phi(s,y)$ be the solution of the system
\begin{equation*}
\dfrac{d t}{d s}=1, \ \
\dfrac{dx_a}{ds}=b_a(t,x), \ \  t(0)=0, \ \ x(0)=y.
\end{equation*}
Then $\epsilon^{k+1}(t,x)=Z(\Phi^{-1}(t,x))$, where $Z(s,y)$ is the solution of
then initial value problem
\begin{equation*}
\dfrac{dZ}{ds}=-c(s,x), \ \ \ \  Z(0,y)=(w^{k+1}_0(y))^{\frac{2}{\gamma-1}}.
\end{equation*}
Obviously
\begin{equation*}
 Z(s,y)=(w^{k+1}_0(y))^{\frac{2}{\gamma-1}}-\int_0^s c(\tau,x(\tau))d\tau.
\end{equation*}
Since $c(\tau,\cdot)\in C_0^\infty$ and
$(w^{k+1}_0)^{\frac{2}{\gamma-1}}\in C_0^\infty$ by Proposition
\ref{prop:6.1},
$Z(s,\cdot)$ also belongs to $C_0^\infty$, and hence
also $Z(\Phi^{-1}(t,\cdot))=\epsilon^ {k+1}
(t,\cdot)=\left(w^{k+1}\right)^{\frac{2}{\gamma-1}}(t, \cdot)\in
C_0^\infty$.
\end{proof}
\vskip 5mm
\noindent
\textbf{Step 3.}
We conclude from Proposition \ref{prop:6.1}, and the theory of linear
symmetric hyperbolic systems (c.f.
\cite{john86:_partial}), that for each $ k$ there is a solution
$U^k(t,x)$ of the linear system (\ref{eq:6.1}) such that
$U^k(t,\cdot)\in C_0^\infty(\setR^3)$.
Therefore by (\ref{eq:6.24}), (\ref{eq:6.9}) and (\ref{eq:6.8}), for
each $k$ we have
\begin{equation*}
 T_k=\sup\left\{T: \sup_{0\leq t\leq T} 
\left\|U^k(t)-U_0^0\right\|_{X_{s,\delta}}\leq R\right\}>0.
\end{equation*}
\begin{prop}
\label{prop:4}
There are  positive constants $T^*$ and $L$ such that
\begin{equation}
\label{eq:6.10}
\sup\left\{T: \sup_{0\leq
t\leq T}\left\|U^k(t)-U_0^0\right\|_{X_{s,\delta}}\leq R
\right\}\geq T^*\ \ \ \text{for all}\ \ k,
\end{equation} 
and
\begin{equation}
\label{eq:6.7}
\sup_{0\leq t\leq T^*}\left\|\partial_t W^k\right\|_{H_{s,\delta+2}}\leq L
 \ \ \ \text{for all}\ \ k.
\end{equation} 
 
\end{prop}
\begin{proof}
We prove it by induction. Set $V^{k+1}=U^{k+1}-U_0^0$, then $V^k$
satisfies the linear initial value problem 
\begin{eqnarray}
\label{eq:6.13}
\left\{\begin{array}{rll}
\mathcal{A}^{0}(v^k,W^k)\partial_t
V^{k+1} &=& \left(\mathcal{A}^a(v^k,W^k)
+\mathcal{C}^a\right)\partial_a
V^{k+1} \\ & + & \mathcal{B}(U^k)
\begin{pmatrix}
v^{k+1}-\phi^0_0\\ \partial_t v^{k+1}-\varphi^0_0\\ \partial_x\left(
v^{k+1}-\phi^{0}_0\right)
\end{pmatrix} + \mathcal{F}(v^k,W^k) + \mathcal{D}^k
\\ V^{k+1}(0,x) & = & U_0^{k+1}(x)-U_0^0(x),
\end{array}\right.
\end{eqnarray}
where
\begin{equation}
\label{eq:6.34}
 \mathcal{D}^k= \left(\mathcal{A}^a(v^k,W^k)+\mathcal{C}^a\right)\partial_a
U^{0}_0  + \mathcal{B}(U^k)\begin{pmatrix}\phi^{0}_0\\ \varphi^{0}_0\\
\partial_x\phi^{0}_0\end{pmatrix}.
\end{equation}

In order to apply the energy estimate, Lemma \ref{lem:1}, to
(\ref{eq:6.13}) we have to check that Assumptions \ref{ass:1} are
satisfied and the corresponding norms are independent of $k$.
Note that all the matrices have the same structure, and from
(\ref{eq:6.6}) we see that condition \ref{eq:5.3} holds.
 By the induction hypothesis,
\begin{equation*}
 \left\|v^k(t)-\phi^0_0\right\|_{H_{s,\delta,2}}^2+
\left\|\partial_x v^k(t)-\partial_x\phi^0_0\right\|_{H_{s,\delta+1,2}}^2\leq
\left\|U^k(t)-U_0^0\right\|_{X_{s,\delta}}^2\leq R^2,
\end{equation*}
therefore by Remark \ref{rem:1},
$\left\|v^k(t)-\phi^0_0\right\|_{H_{s+1,\delta}}\lesssim R$. Applying the
Moser type estimate, Proposition \ref{Moser}, we have
\begin{equation*}
 \begin{split}
&\left\|\mathcal{A}^0(v^k,W^k)-{\bf
e}_{55}\right\|_{H_{s+1,\delta}}^2\lesssim\left(
\left\|v^k\right\|_{H_{s+1,\delta}}^2 +
\left\|W^k\right\|_{H_{s+1,\delta}}^2\right) \\ \lesssim &
\left\|U^k - U_0^0\right\|_{X_{s,\delta}}^2 +
\left\| U_0^0\right\|_{X_{s,\delta}}^2\lesssim\left(R^2+\left\|
U_0^0\right\|_{X_{s,\delta}}^2\right).
 \end{split}
\end{equation*}

In a similar way we get that
$\left\|\mathcal{A}^a(v^k,W^k)\right\|_{H_{s+1,\delta}}$ is bounded by
a constant depending on $R$.
By Propositions \ref{monoton}, \ref{Algebra} ,\ref{Moser}, Remark
\ref{rem:2} and the structure of the matrix $\mathcal{B}$ in
(\ref{eq:3.6}), we obtain that $\left\|{\bf
    b}_{2j}(U^k)\right\|_{H_{s,\delta+1}}\lesssim
\|U^k\|_{X_{s,\delta}}\lesssim\left( R+
  \|U^0_0\|_{X_{s,\delta}}\right)$ and a similar estimate holds for
$\left\|{\bf b}_{4j}(U^k)\right\|_{H_{s,\delta+1}}$, $j=2,3,4,5$.
We recall the $\mathcal{F}(v^k,W^k)=(0,f(v^k,W^k),0,0)$, where
$f(v^k,W^k)$ is given by (\ref{eq:3.10}).
Applying again Propositions \ref{Algebra}, \ref{Moser} and Remark
\ref{rem:2}, we obtain that
\begin{equation}
\label{eq:6.27}
\begin{split}
 \left\| f(v^k,W^k)\right\|_{H_{s,\delta+1}} \lesssim
\left\|(w^k)^{\frac{2}{\gamma-1}}\right\|_{H_{s,\delta+1}}\left(\left\|
v^k\right\|_{H_{s,\delta+1}} + \left\|
(u^\alpha)^k-e_0^\alpha\right\|_{H_{s,\delta+1}}+1\right).
\end{split}
\end{equation} 
Now, for $\frac{3}{2}<s<\frac{2}{\gamma-1}+\frac{1}{2}$, we apply
the estimate for the fractional power, Proposition \ref{Kateb},
and obtain that
\begin{equation}
\label{eq:6.28}
 \left\|(w^k)^{\frac{2}{\gamma-1}}\right\|_{H_{s,\delta+1}}\lesssim
\left\|w^k\right\|_{H_{s,\delta+1}}\leq
\left\|w^k\right\|_{H_{s+1,\delta+1}}. 
\end{equation} 
Since $\left\|v^k\right\|_{H_{s,\delta+1}}$ and $\left\|
W^k\right\|_{H_{s,\delta+1}}$ are bounded by a constant independent of $k$, it
follows from \protect{ (\ref{eq:6.27}) and (\ref{eq:6.28})} that also
 $\left\| f(v^k,W^k)\right\|_{H_{s,\delta+1}}$ is bounded by a constant
independent of $k$.

The required estimate for $\mathcal{D}^k$ defined in (\ref{eq:6.34}) follows
from the
multiplicity property (\ref{eq:elliptic_part:12}) and the estimates
which we have already obtained for
$\left\|\mathcal{A}^\alpha(v^k,W^k)\right\|_{H_{s+1,\delta}}$ and
$\left\|\mathcal{B}(U^k)\right\|_{H_{s,\delta+1}}$.

It remains to verify (\ref{eq:5.9}); note that by the induction
hypothesis (\ref{eq:6.10}), condition (\ref{eq:6.9}) and the embedding
(\ref{eq:em:2}), we have
\begin{equation*}
 \begin{split}
  &\left\| \partial_t\mathcal{A}^0(v^k,W^k)\right\|_{L^\infty}\leq
\sup_{G}\Big|\frac{\partial \mathcal{A}^0}{\partial v}(v,W)\Big|\left\|
\partial_t v^k\right\|_{L^\infty}\\ +\sup_{G} & \Big|\frac{\partial
\mathcal{A}^0}{\pa
W}(v,W)\Big|\left\|
\partial_t W^k\right\|_{L^\infty}\lesssim\left(\left\|
\partial_t v^k\right\|_{H_{s,\delta+1}}+
\left\|\partial_t W^k\right\|_{H_{s,\delta+2}}\right).
\end{split}
\end{equation*}
Since $\left\|\partial_t W^k\right\|_{H_{s,\delta+2}}$ is bounded by
hypothesis (\ref{eq:6.7}), we see that $\left\|
  \partial_t\mathcal{A}^0(v^k,W^k)\right\|_{L^\infty}$  is also bounded
by a constant depending
on $R$ but not on $k$.
We now apply Lemma \ref{lem:1}, and with the combination of Gronwall's
inequality, condition (\ref{eq:6.8}) and the equivalence
(\ref{eq:5.4}), we conclude that there is a constant $C$ depending on
$R$ such that
\begin{equation}
\label{eq:6.26}
\sup_{0\leq t\leq T} \left\|V^{k+1}(t)\right\|_{X_{s,\delta}}^2\leq
e^{Cc_0T}\left(\frac{R^2}{8}+Cc_0^2T\right).
\end{equation}

We turn now to show (\ref{eq:6.7}).
It follows from the structure of the matrices $\mathcal{A}^0$,
$\mathcal{A}^a$ and $\mathcal{B}$ (see (\ref{eq:3.7}), (\ref{eq:3.9})
and (\ref{eq:3.6})) that
\begin{equation}
\label{eq:6.11}
\begin{split}
\partial_t W^{k+1}&= \left( {\bf a}_{44}^0(v^k,W^k)\right)^{-1}
\left[ \sum_{a=1}^3 {\bf a}_{33}^a(v^k,W^k)\partial_a W^{k+1}\right. \\ & \left.
+  {\bf b}_{42}(U^k)\partial_t v^{k+1}+\sum_{a=1}^3{\bf
b}_{4(a+2)}(U^k)\partial_a v^{k+1}\right].
\end{split}
\end{equation}
Note that $\|v^k(t)\|_{H_{s+1,\delta}}$,
$\|W^k(t)\|_{H_{s+1,\delta+1}}$ and $\|U^k(t)\|_{X_{s,\delta}}$ are
bounded by a constant independent of $k$, while $\|\partial_t
v^{k+1}(t)\|_{H_{s,\delta+1}}$, $\|\partial_a
v^{k+1}(t)\|_{H_{s,\delta+1}}$ and $\|\partial_a
W^{k+1}(t)\|_{H_{s,\delta+2}}$ are bounded by (\ref{eq:6.26}).
Hence  applying the  calculus of  the $H_{s,\delta}$--spaces to
(\ref{eq:6.11}), we obtain that  $\left\|\partial_t
  W^{k+1}\right\|_{H_{s,\delta+2}}$ is also bounded by a constant
independent of $k$.
Choosing $T^\ast$ so that
\begin{equation*}
 e^{Cc_0T^\ast}\left(\frac{R^2}{8}+Cc_0^2T^\ast\right)<R^2
\end{equation*}
completes the proof of the Proposition.
\end{proof}
\vskip 5mm
\noindent
\textbf{Step 4.}
Here  we show contraction in  the weighted $L^2$--norm. Our method
relies on the $L^2_\delta$-energy estimates, Lemma \ref{lem:2}
\begin{prop}
\label{prop:6.3}
 There exists positive constants $T^{\ast\ast}\leq T^\ast$ and $\Lambda<1$, and 
a positive sequence $\{\beta_k\}$ with $\sum \beta_k<\infty$ such that 
\begin{equation}
\label{eq:6.17}
\sup_{0\leq t\leq
T^{\ast\ast}}\left\|U^{k+1}(t)-U^k(t)\right\|_{Y_\delta,\mathcal{A}^0}\leq
\Lambda
\sup_{0\leq t\leq
T^{\ast\ast}}\left\|U^{k}(t)-U^{k-1}(t)\right\|_{Y_\delta,\mathcal{A}^0}
+\beta_k.
\end{equation}
\end{prop}
\begin{proof}
 The function $(U^{k+1}-U^k)$ satisfies the linear system
\begin{equation}
\label{eq:6.18}
\begin{split}
 \mathcal{A}^0(v^k,W^k)\partial_t (U^{k+1}-U^k)
& =\mathcal{A}^a(v^k,W^k)\partial_a (U^{k+1}-U^k)
+\mathcal{B}(U^k)\begin{pmatrix} v^{k+1}-v^k\\ \partial_t(
v^{k+1}-v^k)\\ \partial_x (v^{k+1}-v^k)\end{pmatrix}
\\ & +\mathcal{F}(v^k,W^k)-\mathcal{F}(v^{k-1},W^{k-1})+
\mathcal{D}^k,
\end{split}
\end{equation} 
where
\begin{equation}
\label{eq:6.15}
\begin{split}
 \mathcal{D}^k&=-\left(\mathcal{A}^0(v^k,W^k)-\mathcal{A}^0(v^{k-1},
W^{k-1})\right)\partial_t
U^k \\ &+
\left(\mathcal{A}^a(v^k,W^k)-\mathcal{A}^a(v^{k-1},W^{k-1})\right)\partial_a
U^k+
\left(\mathcal{B}(U^k)-\mathcal{B}(U^{k-1})\right)\begin{pmatrix} v^k\\
\partial_t v^k\\ \partial_x v^k\end{pmatrix}.
\end{split}
\end{equation}

By (\ref{eq:6.10}), (\ref{eq:6.7}) and Proposition \ref{Embedding},
$\|\mathcal{A}^\alpha(v^k,W^k)\|_{L^\infty}$,
$\|\partial_\alpha\mathcal{A}^\alpha(v^k,W^k)\|_{L^\infty}$ and 
$\left\|\mathcal{B}(U^k)\right\|_{L^\infty}$ are bounded by a constant
independent of $k$, so we may apply
 Lemma \ref{lem:2} and get that 
\begin{equation}
\label{eq:6.19}
\begin{split}
 &\frac{d}{dt}\left\|U^{k+1}(t)-U^k(t)\right\|_{Y,\mathcal{A}^0}^2\leq C_1\left(
\left\|U^{k+1}(t)-U^k(t)\right\|_{Y,\mathcal{A}^0}^2
\right. \\ & \left.
+\left\|\mathcal{F}(v^k,W^k)-\mathcal{F}(v^{k-1},W^{k-1})\right\|_{L_{ \delta+1
}}^2 + \left\|\mathcal{D}^k\right\|_{L_{
\delta+1
}}^2\right).
\end{split}
\end{equation}

Thus our main task is to estimate the two last terms of the above
inequality by means of the difference
$\|U^{k}-U^{k-1}\|_{Y_\delta}^2$.
Recall that $\mathcal{F}(v,W)=(0,f(v,w),0,0)^T$, where $f$ is a scalar
function given in (\ref{eq:3.10}), so
\begin{equation*}
\begin{split}
 & \ \ \ \ f(v^k,W^k)-f(v^{k-1},W^{k-1}) \\ &=\left(\frac{\partial f}{\pa
v}\left(\tau v^k+\left(1-\tau\right) v^{k-1}, \tau W^k+\left(1-\tau\right)
W^{k-1}\right)\right)\left(v^k-v^{k-1}\right)\\ &
+\left(\frac{\partial f}{\partial W}\left(\tau v^k+\left(1-\tau\right) v^{k-1},
\tau W^k+\left(1-\tau\right)
W^{k-1}\right)\right)\left(W^k-W^{k-1}\right)
\end{split}
\end{equation*}
for some $\tau\in [0,1]$.
Note that $(v^k-v^{k-1})$ belong to $L^2_\delta$, while we need to
estimate the $L^2_{\delta+1}$-norm of the above expressions.
However, since the Makino $w\in H_{s+1,\delta+1}\subset
H_{s,\delta+1}$, we get from (\ref{eq:3.10}), and Propositions
\ref{Kateb}, \ref{Algebra} and \ref{Moser} that $\frac{\partial
  f}{\partial v}\in H_{s,\delta+1}$.
Therefore, by the equivalence $\|u\|_{L^2_\delta}\simeq
\|u\|_{H_{0,\delta}}$ (Proposition \ref{equivalence}) and the
multiplicity property (\ref{eq:elliptic_part:12}), we obtain
\begin{equation*}
\begin{split}
\left\|\frac{\partial f}{\partial
v}\left(v^k-v^{k-1}\right)\right\|_{H_{0,\delta+1}}
\lesssim 
\left\|\frac{\partial f}{\pa
v}\right\|_{H_{s,\delta+1}}\left\|v^k-v^{k-1}\right\|_{H_{0,\delta}
}
\lesssim \left\|\frac{\partial f}{\pa
v}\right\|_{H_{s,\delta+1}}\left\|v^k-v^{k-1}\right\|_{L^2_\delta}.
\end{split}
\end{equation*}

The other term is somewhat easier to treat, since $(W^k-W^{k-1})\in
L^2_{\delta+1}$ and hence 
\begin{equation*}
 \left\|\frac{\partial f}{\partial
W}\left(W^k-W^{k-1}\right)\right\|_{L^2_{\delta+1}}\leq \left\|\frac{\partial
f}{\partial W}\right\|_{L^\infty} \left\|W^k-W^{k-1}\right\|_{L^2_{\delta+2}}.
\end{equation*}
Now $\left\|\frac{\partial f}{\partial W}\right\|_{L^\infty} $ is
bounded by a constant depending on $\left\|U^k\right\|_{X_{s,\delta}}$
and $\left\|U^{k-1}\right\|_{X_{s,\delta}}$, and these are independent
of $k$.
Thus
\begin{equation}
\label{eq:6.23}
 \left\|\mathcal{F}(v^k,W^k)-\mathcal{F}(v^{k-1},W^{k-1})
\right\|_{L^2_{\delta+1}}\leq C_2 \left\|U^k-U^{k-1}\right\|_{Y,\delta},
\end{equation} 
where the constant $C_2$ is independent of $k$.
We shall now estimate the first term of $\mathcal{D}^k$ in (\ref{eq:6.15}). From
the structure of $\mathcal{A}^0(v,W)$ we see that
\begin{equation*}
\begin{split}
& \left(\mathcal{A}^0(v^k,W^k)-\mathcal{A}^0(v^{k-1},W^{k-1})\right)\partial_t
U^k =\\ &\left({\bf a}_{33}^0(v^k)-{\bf
a}_{33}^0(v^{k-1})\right)\partial_t\partial_x v^k +
\left({\bf a}_{44}^0(v^k,W^k)-{\bf
a}_{44}^0(v^{k-1},W^{k-1})\right)\partial_tW^k.
\end{split}
\end{equation*}

Now
\begin{equation*}
{\bf a}_{33}^0(v^k)-{\bf a}_{33}^0(v^{k-1})=\frac{\partial{\bf a}_{33}^0}{\pa
v}\left(\tau v^k+(1-\tau)v^{k-1}\right)\left(v^k-v^{k-1}\right)
\end{equation*} 
for some $\tau\in[0,1]$.
But since $(v^k-v^{k-1})\not\in L^2_{\delta+1}$, we cannot use the
$L^\infty$--$L^2$ estimates.
We therefore apply the algebra (or multiplication) property, once 
with $s_1=1$, $s_2=s-1$ and $s=0$, and once with $s_1=s$, $s_2=1$ and
$s=1$, which results in the following inequality:
\begin{equation}
\label{eq:6.29}
 \begin{split}
&
\left\|\left({\bf a}_{33}^0(v^k)-{\bf
a}_{33}^0(v^{k-1})\right)\partial_t\partial_x v^{k-1}\right\|_{H_{0,\delta+1}}\\
\lesssim &
\left\| \frac{\partial{\bf a}_{33}^0}{\pa
v}\left(\tau
v^k+(1-\tau)v^{k-1}\right)\left(v^k-v^{k-1}\right)\right\|_{H_{1,\delta}}
\left\|\partial_t\partial_x v^{k-1}\right\|_{H_{s-1,\delta+2}}
\\ \lesssim &
\left\| \frac{\partial{\bf a}_{33}^0}{\pa
v}\left(\tau
v^k+(1-\tau)v^{k-1}\right)\right\|_{H_{s,\delta}}
\left\|v^k-v^{k-1}\right\|_{H_{1,\delta } } 
\left\|\partial_t v^{k-1}\right\|_{H_{s,\delta+1}}.
 \end{split}
\end{equation}

We use now the third Moser
inequality (\ref{eq:14}) and $\|v^k-v^{k-1}\|_{L^\infty}\lesssim
\|v^k-v^{k-1}\|_{H_{s+1,\delta}}$ in order to obtain
\begin{equation}
\label{eq:6.30}
 \left\| \frac{\partial{\bf a}_{33}^0}{\partial
v}\left(\tau v^k+(1-\tau)v^{k-1}\right)\right\|_{H_{s+1,\delta}}\leq
C
\left(\left\| v^k\right\|_{H_{s+1,\delta}}, 
\left\|v^{k-1}\right\|_{H_{s+1,\delta}}\right).
\end{equation}
where 
\begin{displaymath}
C
\left(\left\| v^k\right\|_{H_{s+1,\delta}}, 
\left\|v^{k-1}\right\|_{H_{s+1,\delta}}\right)
\end{displaymath}
denotes that the
constant depends in some way on the terms
$\| v^k\|_{H_{s+1,\delta}}$ and 
$\left\|v^{k-1}\right\|_{H_{s+1,\delta}}$.

In a similar way we obtain
\begin{equation}
\label{eq:6.31}
 \begin{split}
   &\left\|\left({\bf a}_{44}^0(v^k,W^k)-{\bf
a}_{44}^0(v^{k-1},W^{k-1})\right)\partial_tW^k\right\|_{L^2_{\delta+1}}
 \\ \lesssim &
\left(\left\|\frac{\partial{\bf a}_{44}^0}{\partial
v}\left(\tau v^k+(1-\tau)v^{k-1}\right)\right\|_{H_{s,\delta+1}}
\left\|v^k-v^{k-1}\right\|_{H_{1,\delta}}\right. \\ & \left.+
\left\|\frac{\partial{\bf a}_{44}^0}{\partial W}\right\|_{L^\infty}
\left\|W^k-W^{k-1}\right\|_{L^2_{\delta+1}}
\right)
\left\|\partial_t W^k\right\|_{H_{s,\delta+2}}.
\end{split}
\end{equation} 
We recall that  $ \left\|\partial_t W^k\right\|_{H_{s,\delta+2}}$ is
bounded by (\ref{eq:6.7}) and
$\left\|v^k-v^{k-1}\right\|_{H_{1,\delta}}^2\simeq
\left\|v^k-v^{k-1}\right\|_{L^2_\delta}^2+\left\|\partial_xv^k-\partial_x
v^{k-1}\right\|_{L^2_{\delta+1} }^2$, 
therefore
from (\ref{eq:6.29}), (\ref{eq:6.30}) and (\ref{eq:6.31}) we obtain that
\begin{equation}
\label{eq:6.21}
 \left\|\left(\mathcal{A}^0(v^k,W^{k})-\mathcal{A}^0(v^{k-1},
W^{k-1})\right)\partial_t
U^k\right\|_{L^2_{\delta+1}}\leq C_3 \left\|
U^k-U^{k-1}\right\|_{Y_{\delta}}.
\end{equation}

In a similar manner we can estimate the difference involving the
$\mathcal{A}^a$ matrices.
The estimate of $\mathcal{B}(U^k)-\mathcal{B}(U^{k-1})$ is simpler,
since its first column of the matrix $\mathcal{B}$ is zero and
therefore this expression does not contain the element
$(v^k-v^{k-1})$.
Thus we conclude from inequalities (\ref{eq:6.19}),
(\ref{eq:6.23}) and (\ref{eq:6.21}) that
\begin{equation*}
\begin{split}
 &\frac{d}{dt}\left\|U^{k+1}(t)-U^k(t)\right\|_{Y,\mathcal{A}^0}^2\leq C_4\left(
\left\|U^{k+1}(t)-U^k(t)\right\|_{Y,\mathcal{A}^0}^2 + 
\left\|U^{k}(t)-U^{k-1}(t)\right\|_{Y,\mathcal{A}^0}^2
\right),
\end{split}
\end{equation*} 
where $C_4$ is independent of $k$. Therefore by Gronwall's
inequality, we obtain that for any
$T^{\ast\ast}\leq T^\ast$, 
\begin{equation*}
\begin{split}
 \sup_{0\leq t
\leq T^{\ast\ast}} &\left\|U^{k+1}(t)-U^k(t)\right\|_{Y,\mathcal{A}^0}^2\leq
e^{C_4T^{\ast\ast}}\left(\left\|U_0^{k+1}-U_0^k\right\|_{Y_\delta,\mathcal{A}^0}
^2 \right. \\ &+ T^{\ast\ast}C_4\left.\sup_{0\leq t
\leq T^{\ast\ast}}\left\|U^{k}(t)-U^{k-1}(t)\right\|_{Y,\mathcal{A}^0}^2\right),
\end{split}
\end{equation*} 
and hence inequality (\ref{eq:6.17}) holds if we chose $T^{\ast\ast}$
so that $\Lambda:=\sqrt{2C_4e^{C_4T^{\ast\ast}}T^{\ast\ast}}<1$, and
set $\beta_k:=\sqrt{2e^{C_4T^{\ast\ast}}}
\left\|U_0^{k+1}-U_0^k\right\|_{Y_\delta,\mathcal{A}^0}$.
\end{proof}
\vskip 5mm
\noindent
\textbf{Step 5.}
We discuss here the convergence.
It follows from Proposition \ref{prop:6.3} that 
\begin{equation*}
\sum_k
\left\|U^{k+1}(t)-U^k(t)\right\|_{Y_\delta}<\infty,
\end{equation*}
hence $\{U^k(t)\}$ is a Cauchy sequence in
$L^\infty\left([0,T^{\ast\ast}],Y_\delta\right)$.
Applying the Gagliardo-Nirenberg-Moser estimate $\|u\|_{H^{s'}}\leq
\|u\|_{H^{s}}^{\frac{s'}{s}}\|u\|_{L^2}^{1-\frac{s'}{s}}$ term--wise
to the norm (\ref{eq:weighted:4}), we get that
\begin{equation}
 \|u\|_{H_{s',\delta}}\leq\|u\|_{H_{s,\delta}}^{\frac{s'}{s}}\|u\|_{L^2_\delta}^
{ 1-\frac{s'}{s}}\quad \text{for}\quad 0<s'<s\ \text{and} \ \delta\in\setR.
\end{equation}

Hence $\{U^k(t)\}$ is a Cauchy sequence in 
$L^\infty\left([0,T^{\ast\ast}],X_{s',\delta}\right)$
and therefore $U^k(t)\to U(t)$ in the $X_{s',\delta}$--norm 
for any $0<s'<s$ and $\delta\geq -\frac{3}{2}$.
Furthermore, by  Remark \ref{rem:2}, $v^k(t)\to
v(t)$ in $H_{s'+1,\delta}$-norm. Thus if we choose $\frac{3}{2}<s'<s$, then by
the embedding (\ref{eq:em:2}),
\begin{equation*}
 v^k(t)\to v(t), \  W^k(t)\to W(t) \quad \text{in}\quad C^1
\end{equation*}
and
\begin{equation*}
 \partial_t v^k(t)\to \partial_t v(t),\  \partial_tW^k(t)\to \partial_tW(t)
\quad \text{in}\quad C^0.
\end{equation*} 
Thus, $U(t)=(v(t),\partial_t v(t),\partial_x v(t),W(t)) $ is a solution of the
system
(\ref{eq:3.4}).

\begin{prop}
\label{prop:5}
 For any $\Phi\in X_{s,\delta}$,
\begin{equation}
\label{eq:6.12}
 \lim_k\left\langle U^k(t),\Phi\right\rangle_{X_{s,\delta}}=\left\langle
U(t),\Phi\right\rangle_{X_{s,\delta}}
\end{equation}
uniformly for $0\leq t\leq T^{\ast\ast}$, and where $\left\langle
  \cdot,\cdot\right\rangle_{X_{s,\delta}}$ denote the inner-product
(\ref{eq:5.6}) with $\mathcal{A}^0$ being the identity matrix.
\end{prop}

As a consequence of the weak convergence (\ref{eq:6.12}), we have that
\begin{equation*}
 \left\|U(t)\right\|_{X_{s,\delta}}\leq \liminf_k
\left\|U^k(t)\right\|_{X_{s,\delta}}.
\end{equation*} 
Thus the solution $U(t)$ belongs to $X_{s,\delta}$.
For the  proof of Proposition \ref{prop:5} see \cite[\S5]{Ka1}.
\vskip 5mm
\noindent
\textbf{Step 6.}  Here we shall prove  uniqueness.
\begin{prop}
 Suppose $U_1(t), U_2(t)\in X_{s,\delta}$ are two solutions of (\ref{eq:3.4})
with the same initial data,  then $U_1(t)\equiv U_2(t)$.
\end{prop}
\begin{proof}
Let $V(t)=U_1(t)-U_2(t)$, then it satisfies the  same type  of a
linear system as
(\ref{eq:6.18}), therefore by similar estimates as in Step 4,
we obtain that
\begin{equation*}
 \dfrac{d}{dt}\left\| V(t)\right\|_{Y_{\delta,\mathcal{A}^0}}^2 \lesssim 
\left\| V(t)\right\|_{Y_{\delta,\mathcal{A}^0}}^2,
\end{equation*}
and since $V(0)\equiv 0$, Gronwall's inequality implies that $V(t)\equiv 0$.
\end{proof}
\vskip 5mm
\noindent
\textbf{Step 7.} 
Since $X_{s,\delta}$ is a Hilbert space it suffices to show that
\begin{equation}
\label{eq:6.14}
 \limsup_{t\to
0^+}\left\|U(t)\right\|_{X_{s,\delta,\mathcal{A}^0}}\leq
\left\|U(0)\right\|_{X_{s,\delta,\mathcal{A}^0}}
\end{equation} 
in order to establish the continuity in the norm.
Here $\mathcal{A}^0$ depends on the initial data $\phi$ and $W_0$,
that is, $\mathcal{A}^0=\mathcal{A}^0(\phi,W_0)$.
The proof of (\ref{eq:6.14}) relies on the same arguments as in
\cite{majda84:_compr_fluid_flow_system_conser} and therefore we leave
it out.
This completes the proof of Theorem \ref{thm:1}.

\section{Proof of the main result}
\label{sec:main}
The proof of the main result, Theorem \ref{thm:main}, actually follows
from Theorem \ref{thm:1}, we just have to check whether the initial
data of the gravitational fields and of the fluid satisfy the
assumptions of Theorem \ref{thm:1}.
We recall that $v(t)=g_{\alpha\beta}(t)-\eta_{\alpha\beta}$, so
setting $\phi=v(0)$, we have by the assumptions of Theorem
\ref{thm:main} that $\phi\in H_{s+1,\delta}$.
The initial data for the time derivative $\varphi$ are given
by $\partial_t g_{\alpha\beta}(0)$, where
$\partial_tg_{ab}(0)=-2K_{ab}$ ($a,b=1,2,3$), and $\partial_t
g_{\alpha0}(0)$ is given by expression (\ref{eq:2.7}).
By the assumption of Theorem \ref{thm:main}, $K_{ab}\in
H_{s,\delta+1}$ and therefore by Propositions \ref{Algebra} and
\ref{Moser}, $\partial_t g_{\alpha0}(0)$ also belongs to
$H_{s,\delta+1}$.
Thus $\varphi=\partial_t g_{\alpha\beta}(0)$ satisfies the initial
condition of Theorem \ref{thm:1}.
Note that ${\bf a}_{33}^0(0)=h^{ab}$, where $h_{ab}$ is a proper
Riemannian metric. Since $w(0)\geq 0$ and $u^\alpha(0)$ is a unit
timelike vector, ${\bf a}_{44}^0(0)$ is a positive definite matrix by
Theorem \ref{thm:2}.
Hence $\mathcal{A}^0(\phi,W_0)$ satisfy
condition (\ref{eq:6.5}) and we conclude that
\begin{math}
  U(t)=\left(g_{\alpha\beta}(t)-\eta_{\alpha\beta}, 
\partial_t g_{\alpha\beta}(t),\partial_x g_{\alpha\beta}(t),W(t)\right)\in
C([0,T],X_{s,\delta})
\end{math}.
Hence (\ref{eq:2.13}) follows from Remark \ref{rem:1}, and
(\ref{eq:2.14}) from (\ref{eq:6.4}).
That completes the proof.

\section{Appendix}
\label{sec:appendix}
The classical paper of Hughes, Kato and Marsden \cite{hughes76:_well}
established the short time existence of the vacuum Einstein equations
by solving a second order quasi--linear hyperbolic system whose
solutions $(g_{\alpha\beta},\partial_t g_{\alpha\beta})$ belong to
$H^{s+1}\times H^s$ for $s>\frac{3}{2}$.

On the other hand, Fisher and Marsden treated the Einstein vacuum
equation by means of the theory of symmetric hyperbolic
systems.
However, they only obtained the regularity of $s>\frac{7}{2}$.
In \cite{Ka1} we generalized the result of \cite{hughes76:_well} to
the $H_{s,\delta}$ spaces, treating however, the Einstein equations as
a symmetric hyperbolic system.
Since the techniques of \cite{Ka1}, and in particular the energy
estimates, play an essential role in the present paper, we outline its
main idea that enables us to obtain the same regularity as in
\cite{hughes76:_well}.

We present a heuristic argument explaining the essential idea.
First, if a function $v$ satisfies a wave equation, then the vector
$V=(v,\partial_t v,\partial_x v)$ satisfies a symmetric hyperbolic
system.
The general condition for existence and uniqueness in the
$H^s(\setR^3)$ spaces is $s>\frac{5}{2}$.
Hence, we have by this method that $ \partial_t v,\partial_x v\in H^s$
for $s>\frac{5}{2}$.

However, in our case we improve this regularity  to
$(v,\partial_tv,\partial_xv)\in H^{s+1}\times H^s\times H^s$ for
$s>\frac{3}{2}$. This is because we do not consider a general quasi--linear
symmetric hyperbolic system where the matrices $A^a(V)$ depend on $V$,
but a system in which the matrices $A^a(v)$ only depend on $v$ but
\textbf{not} on its derivatives.

In order to see how this fact allows us to improve
the regularity of the solution we will derive energy estimates for the
linearized symmetric hyperbolic system.
For the sake of clarity we consider a simple hyperbolic system
\begin{equation*}
\partial_t V=A^a(v)\partial_a V,
\end{equation*}
then its linearized form is
\begin{equation}
\label{eq:sec8-appendix:11}
\partial_t V=\tilde A^a\partial_aV.
\end{equation}
Note that in each iteration we solve the linear system
(\ref{eq:sec8-appendix:11}) with $
\widetilde{A}^a=\widetilde{A}^a(v^k)$, and since $V^k=(v^k,\partial_t
v^k,\partial_x v^k)\in H^s$, $v^k\in H^{s+1}$ and hence $
\widetilde{A}^a=\widetilde{A}^a(v^k)\in H^{s+1}$ by Moser type
estimates.
The crucial step is to derive the energy estimate
\begin{equation}
\label{eq:appen:8}
 \dfrac{d }{dt} \left(\dfrac{1}{2}\left\|V\right\|_{H^s}^2\right)\leq C
\left\|V\right\|_{H^s}^2
\end{equation}
for $s>\frac{3}{2}$ and whenever $V$ satisfies the linear system
(\ref{eq:sec8-appendix:11}).
We recall that $\|V\|_{H^s}=\|\Lambda^s V\|_{L^2}$, where $\Lambda^s$
is the pseudodifferential operator $(1-\Delta)^{\frac{s}{2}}$.

One of the basic tools for obtaining (\ref{eq:appen:8}) are
the commutator's estimates.
Here we shall use the following Pseudodifferential operators version
of the  Kato--Ponce estimate \cite[\S3.6]{Taylor91}: Let $P$ be a
differential operator in the class $OPS^s_{1,0}$, then
\begin{equation}
\label{eq:appen:9}
 \left\|P(fg)-fP(g)\right\|_{L^2}\leq C\left\{\|\nabla
f\|_{L^\infty}\|g\|_{H^{s-1}}+\|f\|_{H^s}\|g\|_{L^\infty}\right\},
\end{equation}
for any $f\in H^s\cap C^1$ and $g\in H^{s-1}\cap L^\infty$.

The standard way to obtain (\ref{eq:appen:8}) is to differentiate
$\left\|V\right\|_{H^s}^2$ with respect to time,
to insert the differential equation (\ref{eq:sec8-appendix:11}) and then apply 
a suitable commutator which leads to
\begin{equation*}
 \begin{split}
\label{eq:sec8-appendix:6} 
\dfrac{1}{2} \dfrac{d }{dt} \left\|V\right\|_{H^s}^2 & =
\left\langle
\Lambda^s(V),\Lambda^s\left(\partial_t V\right)\right\rangle_{L^2}=\left\langle
\Lambda^s(V),\Lambda^s\left(\widetilde{A}^a\partial_a
V\right)\right\rangle_{L^2}
\\ & = \left\langle
\Lambda^s(V),\widetilde{A}^a\left(\Lambda^s\left(\partial_a
V\right)\right)\right\rangle_{L^2}
\\ & + \left\langle
\Lambda^s(V),\left[ \Lambda^s\left(\widetilde{A}^a\partial_a
V\right)-\widetilde{A}^a\left(\Lambda^s\left(\partial_a
V\right)\right)\right]\right\rangle_{L^2},
 \end{split}
\end{equation*}
and then the first term is taken care
of by integration by parts and the second one is by applying
the above Kato--Ponce estimate to the operator $\Lambda^s$.
But this procedure results in a term of the form $\left\| \partial_a V
\right\|_{L^{\infty}}$ which contains $\|\partial_a\partial_x
v\|_{L^\infty}$.
In order to estimate it by $ \|\partial_a\partial_x v\|_{H^{s-1}}
\lesssim \|\partial_x v\|_{H^{s}}$ we need to require that
$s-1>\frac{3}{2}$, and hence we do not get the desired result.

We circumvent this difficulty by writing
\begin{equation*}
 \widetilde{A}^a\partial_a V=\partial_a\left( \widetilde{A}^a
V\right)-\left(\partial_a \widetilde{A}^a \right)V,
\end{equation*}
and making the commutation  
\begin{equation*}
 \Lambda^s\left(\partial_a\left( \widetilde{A}^a
V\right)\right)= \left[\left(\Lambda^s\partial_a\right) \left(\widetilde{A}^a
V\right)-  \widetilde{A}^a\left(\Lambda^s\partial_a\right)\left(
V\right) \right]+ \widetilde{A}^a\left(\Lambda^s\partial_a\right)\left(
V\right).
\end{equation*}
which we insert into the first row of equation
(\ref{eq:sec8-appendix:6}).
Then we have to estimate three terms:
\begin{equation*}
I= \left\langle
\Lambda^s(V),\left[\left(\Lambda^s\partial_a\right) \left(\widetilde{A}^a
V\right)-  \widetilde{A}^a\left(\Lambda^s\partial_a\right)\left(
V\right) \right]\right\rangle_{L^2},  
\end{equation*} 
\begin{equation*}
II= \left\langle
\Lambda^s(V), \widetilde{A}^a\left(\Lambda^s\partial_a\right)\left(
V\right) \right\rangle_{L^2}
\end{equation*} 
and
\begin{equation*}
III=
 \left\langle
\Lambda^s(V),\Lambda^s\left((\partial_a
\widetilde{A}^a)V\right)\right\rangle_{L^2}.
\end{equation*}

For the first term we apply the Kato-Ponce commutator
(\ref{eq:appen:9}).
However, this time we do it for the operator $(\Lambda^s\partial_a)$
which has order $s+1$, and hence
\begin{equation*}
\begin{split}
 |I|&\leq \left\|V\right\|_{H^s}\left\|\left(\Lambda^s\partial_a\right)
\left(\widetilde{A}^a
V\right)-  \widetilde{A}^a\left(\Lambda^s\partial_a\right)\left(
V\right)\right\|_{L^2}\\ &\lesssim
\left\|V\right\|_{H^s}\left\{\|\nabla
\widetilde{A}^a\|_{L^\infty}\|V\|_{H^s}+\|\widetilde{A}^a\|_{H^{s+1}}\|V\|_{
L^\infty}\right\}.
\end{split}
\end{equation*}
So by Sobolev embedding theorem, we see that $|I|\lesssim \|
\widetilde{A}^a\|_{H^{s+1}}\|V\|_{H^s}^2$.
Likewise, since $H^s$ is an algebra for $s> \frac{3}{2}$,
\begin{equation*}
 |III|\lesssim \|V\|_{H^s}\|(\partial_a
\widetilde{A}^a)V\|_{H^s}\lesssim \|V\|_{H^s}^2\|(\partial_a
\widetilde{A}^a)\|_{H^s}\lesssim \|\widetilde{A}^a\|_{H^{s+1}}\|V\|_{H^s}^2.
\end{equation*}

Since $\Lambda^s\partial_a=\partial_a\Lambda^s$ and $\widetilde{A}^a$
is symmetric, we obtain a similar estimate for $II$ by using
integration by parts.
Hence we conclude that the energy estimate (\ref{eq:appen:8}) holds.
Note that in the estimate of all three terms above we have used the
fact that $ \widetilde{A}^a\in H^{s+1}$.

For the general case where $A^0\neq I$, one has to define an
appropriated inner-product which takes into account the matrix $A^0$.
Details for the vacuum equations in the weighted spaces $H_{s,\delta}$
and a positive definite $A^0$ can be found in \cite[\S4]{Ka1} and only
slight modifications are needed in order to extend the energy
estimates of \cite{Ka1} to the coupled system (\ref{eq:3.4}).

\bibliographystyle{amsplain}
\bibliography{bibgraf}

\end{document}